\RequirePackage{fix-cm}
\documentclass[smallextended]{svjour3}       
\smartqed  
\usepackage{graphicx}
\usepackage{enumerate}
\usepackage{braket,amsfonts,amsopn,amsmath,yhmath}
\usepackage{amscd}
\usepackage{algorithm}
\usepackage{pifont}
\usepackage{tabularx,booktabs}
\usepackage{color}
\usepackage{graphicx}
\usepackage{subcaption}
\usepackage{url}

\DeclareMathOperator{\Graph}{Graph}

\newcommand{\innerP}[2]{\left\langle {#1},{#2} \right\rangle}
\newcommand{\dist}{\mathrm{dist}}

\newcommand{\sign}{\mathrm{sign}}
\newcommand{\dom}{\mathrm{dom}}
\newcommand{\prox}{\mathrm{prox}}

\newcommand{\Liminf}{\mathop{\lim\inf}}

\newcommand{\One}{\mathbf{1}}

\newcommand{\StatementNum}[1]{{\upshape(\romannumeral #1 )}}
\newcommand{\StatementNumUpperCase}[1]{{(\uppercase\expandafter{\romannumeral #1})}}

\newcommand{\wF}{\widetilde{F}}

\newcommand{\AlphaStar}{\min\{\underline{\alpha},\underline{\alpha}_{\sigma} \}\gamma}

\newtheorem{assumption}{Assumption}

\newtheorem{lemma}[theorem]{Lemma}

\newtheorem{proposition}[theorem]{Proposition}

\newtheorem{corollary}[theorem]{Corollary}

\allowdisplaybreaks

\begin{document}
	
\title{A min-max reformulation and proximal algorithms for a class of structured nonsmooth fractional optimization problems\thanks{Junpeng Zhou and Na Zhang contributed equally to this paper.
}}

\titlerunning{Min-Max Reformulation and AMPDA for FP}        

\author{Junpeng Zhou         \and Na Zhang \and Qia Li
}


\institute{Junpeng Zhou \at
              School of Computer Science and Engineering, Sun Yat-sen University, Guangzhou 510275, China \\
              \email{zhoujp5@mail2.sysu.edu.cn}           
           \and
           Na Zhang \at
              Department of Applied Mathematics, College of Mathematics and Informatics, South China Agricultural University, Guangzhou 510642, China\\
              \email{nzhsysu@gmail.com}
           \and
           Qia Li \at
           School of Computer Science and Engineering, Guangdong Province Key Laboratory of Computational Science, Sun Yat-sen University, Guangzhou 510275, China\\
           Corresponding author. \\
           \email{liqia@mail.sysu.edu.cn}
}

\date{Received: date / Accepted: date}

\maketitle

\begin{abstract}
In this paper, we consider a class of structured nonsmooth fractional minimization, where the first part of the objective is the ratio of a nonnegative nonsmooth nonconvex function to a nonnegative nonsmooth convex function, while the second part is the difference of a smooth nonconvex function and a nonsmooth convex function. This model problem has many important applications, for example, the scale-invariant sparse signal recovery in signal processing. However, the existing methods for fractional programs are not suitable for solving this problem due to its special structure. We first present a novel nonfractional min-max reformulation for the original fractional program and show the connections between their global (local) optimal solutions and stationary points. Based on the reformulation, we propose an alternating maximization proximal descent algorithm and show its subsequential convergence towards a critical point of the original fractional program under a mild assumption. Moreover, we prove that the proposed algorithm can find an $\epsilon$-critical point of the considered problem within $\mathcal{O}(\epsilon^{-2})$ iterations. By further assuming the Kurdyka-{\L}ojasiewicz (KL) property of an auxiliary function, we also establish the convergence of the entire solution sequence generated by the proposed algorithm. Finally, some numerical experiments on the $L_1/L_2$ least squares problem and scale-invariant sparse signal recovery are conducted to demonstrate the efficiency of the proposed method.
\end{abstract}


\section{Introduction}
\label{section:introduction}
In this paper, we consider the following structured nonsmooth fractional minimization problem
\begin{equation}\label{problem:root}
	\min \left\{\frac{f(x)}{g(x)} + h(x):x\in\Omega\cap\mathcal{C} \right\},
\end{equation}
where $f, g: \mathbb{R}^n \rightarrow[0,+\infty)$ and $h: \mathbb{R}^n \rightarrow \mathbb{R}$ are proper closed, $\mathcal{C} \subseteq \mathbb{R}^n$ is closed and $\mathcal{C} \cap \Omega \neq \emptyset$ with $\Omega:=\left\{x \in \mathbb{R}^n: g(x) \neq 0\right\}$. Let $f_{\mathcal{C}}: \mathbb{R}^n \to [0,+\infty]$ be the sum of $f$ and the indicator function on $\mathcal{C}$.
Throughout this paper, we adopt the following blanket assumption for problem \eqref{problem:root}.

\begin{assumption}
	\indent
	\begin{enumerate}[\upshape(\romannumeral 1 )]
		\item $f$ is locally Lipschitz continuous on $\Omega$.
		\item $g$ is convex on $\mathbb{R}^n$.
		\item $h:=h_1-h_2$, where $h_1: \mathbb{R}^n \rightarrow \mathbb{R}$ is locally Lipschitz differentiable around each $x \in \Omega$ and $h_2: \mathbb{R}^n \rightarrow \mathbb{R}$ is convex on $\mathbb{R}^n$.
		\item The proximal operator associated with $f_{\mathcal{C}}$ can be evaluated.\footnote{The proximal operator associated with $f_{\mathcal{C}}$ is defined as $\prox_{f_{\mathcal{C}}}(x):= \arg\min\{ f_{\mathcal{C}}(z)+\frac{1}{2}\|x-z\|^2_2: z\in\mathbb{R}^n \}.$}
	\end{enumerate}
\end{assumption}

Problem \eqref{problem:root} covers many important optimization models in machine learning and scientific computing. In this work, we focus on its application to the scale-invariant sparse signal recovery \cite{Rahimi-Wang-Dong-Lou:2019SIAM-SC,TaoMin:2022SIAM-SC,Wang-Aujol-Gilboa-Lou:2024IPI,Wang-Tao-Nagy-Lou:2021SIAM-ImageScience,Wang-Yan-Rahimi-Lou:2020IEEE,Yin-Esser-Xin:2014CIS,Zeng-Yu-Pong:2021SIAM-OPT,Zhang-Liu-Li:2025ACHA,NaZhang-QiaLi:2022SIAM-OPT}. Let $A\in\mathbb{R}^{m\times n}$ be the sensing matrix, $b\in\mathbb{R}^m$ be the possibly noisy measurement, $\underline{x}\in\mathbb{R}^n$ and $\overline{x}\in\mathbb{R}^n$ be the lower bound and upper bound of the underlying signal, respectively. We are particularly interested in the $L_1/L_2$ robust signal recovery model
\begin{equation}\label{problem: L1dL2}
	\min \left\{\frac{\|x\|_{1}}{\|x\|_{2}}+\frac{\lambda}{2} \dist^2 \left(A x-b, \mathcal{S}_{\mu}\right): x \neq 0,~ \underline{x} \leq x \leq \overline{x},~ x \in \mathbb{R}^{n}\right\},
\end{equation}
and the $L_1/S_K$ robust signal recovery model
\begin{equation}\label{problem: L1dSK}
	\min \left\{\frac{\|x\|_{1}}{\|x\|_{(K)}}+\frac{\lambda}{2} \dist^2(A x-b, \mathcal{S}_{\mu}): x \neq 0,~ \underline{x} \leq x \leq \overline{x},~ x \in \mathbb{R}^{n}\right\},
\end{equation}
where $\|x\|_{(K)}$ denotes the sum of the $K$ largest absolute values of entries in $x$ for the positive integer $K$, $\lambda>0$, $\mathcal{S}_{\mu}:=\{z \in \mathbb{R}^{m}:\|z\|_{0} \leq \mu \}$ with $\mu$ being a nonnegative integer and $\|\cdot\|_{0}$ counting the number of nonzero elements in the vector.
A direct computation yields that $\dist^2(A x-b, \mathcal{S}_{\mu}) = \|Ax-b\|^2_2 - \|\mathcal{T}_{\mu}(Ax-b)\|^2_2$ with $\mathcal{T}_{\mu}(z)$ being the projection operator onto $\mathcal{S}_{\mu}$, which retains the $\mu$ largest elements of the vector $z\in\mathbb{R}^m$ in absolute value and sets the rest to be zero. Moreover, it is not hard to verify that $\|\mathcal{T}_{\mu}(Ax-b)\|^2_2$ is convex. Then, it is obvious that both problems \eqref{problem: L1dL2} and \eqref{problem: L1dSK} are special examples of problem \eqref{problem:root} with $\mathcal{C}$ being the indicator function on $\{x \in \mathbb{R}^{n}: \underline{x} \leq  x \leq  \overline{x}\}$, $f(x)=\|x\|_{1}$, $g(x)=\|x\|_{2}$ or $\|x\|_{(K)}$, $h_{1}(x)=\frac{\lambda}{2}\|A x-b\|^{2}_2$, and $h_{2}(x)=\frac{\lambda}{2}\|\mathcal{T}_{\mu}(A x-b)\|_{2}^{2}$.

Now we briefly review two categories of approaches for solving structured nonsmooth fractional optimization problems as follows. 

The first category of approaches is based on the classical Dinkelbach's method \cite{Crouzeix-Ferland-Schaible:1985JOTA,Dinkelbach-Werner:MS1967,Ibaraki:1983Mathematical_Programming,Schaible:1976Fractional_Programming}, which is widely used for single-ratio fractional programs. Note that problem \eqref{problem:root} can be naively converted into a signal-ratio optimization problem with the numerator $f+hg$, for which the Dinkelbach's method generates the new iterate $x^{k+1}$ by solving
\begin{equation}\label{eq: z09272035}
	\min \left\{f(x)+ (hg)(x) - \theta_k g(x): x\in\mathcal{C} \right\},
\end{equation}
where $\theta_k\in\mathbb{R}$ is updated via $\theta_k:=f(x^k) / g(x^k)+h(x^k)$.
This iterative scheme is not practical in general since it may be very difficult and expensive to solve an optimization problem of type \eqref{eq: z09272035} in each iteration. Recently, proximal algorithms based on the Dinkelbach's method have been developed for a class of nonsmooth single-ratio fractional programs \cite{Bot-Csetnek:2017Optimization,Bot-Dao-Li:2021MathematicsofOperationsResearch,Li-Shen-Zhang-Zhou:2022ACHA,NaZhang-QiaLi:2022SIAM-OPT,Zhang-Liu-Li:2025ACHA,Zhou-Zhang-Li:2024MOR}. 
If we ignore the line-search and extrapolation involved, then applying these algorithms to problem \eqref{problem:root} yields the iteration scheme
\begin{equation}\label{eq: z12131419}
	x^{k+1} \in \arg\min\left\{ f(x)+(hg)(x)-\left\langle \theta_{k} y^{k}, x-x^{k}\right\rangle + \frac{\|x-x^{k}\|_{2}^2}{2 \eta_{k}}: x \in \mathcal{C}\right\},
\end{equation}
or
\begin{equation}\label{eq: z12131420}
	x^{k+1} \in \arg\min\left\{f(x)+\left\langle\nabla (hg)(x^{k})- \theta_{k} y^{k}, x-x^{k}\right\rangle+\frac{\|x-x^{k}\|_{2}^2}{2 \eta_{k}}: x \in \mathcal{C}\right\}, 
\end{equation}
where $y^{k} \in \partial g(x^{k})$, $\eta_k>0$ is some stepsize, and $\theta_{k}$ is defined as in \eqref{eq: z09272035}. It is worth noting that \eqref{eq: z12131419} actually requires that the proximal operator associated with the sum of $f+hg$ and indicator function on $\mathcal{C}$ can be evaluated, while \eqref{eq: z12131420} assumes the smoothness of $hg$. Unfortunately, the above requirements do not hold in general and consequently the iteration schemes \eqref{eq: z12131419} and \eqref{eq: z12131420} are not suitable for tackling problem \eqref{problem:root}. Therefore, we conclude that the Dinkelbach's method and its existing variants are not applicable to solving problem \eqref{problem:root}.

The second category of approaches is based on the so-called quadratic transform \cite{Benson:2004JOTA,Bot-Dao-Li:2023SIAM-OPT,Shen-Yu:2018IEEE}. 
Let $\varphi_i, \psi_i : \mathbb{R}^n \to [0,+\infty)$, $i=1,2,\ldots,N$, be proper closed functions.
The quadratic transform refers to reformulating the problem of maximizing a sum of ratios $\sum_{i=1}^{N} \frac{\varphi_i(x)}{\psi_i(x)}$ into the one of maximizing problem $\sum_{i=1}^{N} 2y_i\sqrt{\varphi_i(x)} - y^2_i\psi_i (x)$ with respect to $y\in\mathbb{R}^N$ and $x$. The equivalence of global optimality is built on the fact that $\max_{y_i} 2y_i\sqrt{\varphi_i(x)} - y^2_i\psi_i (x) = \frac{\varphi_i(x)}{\psi_i(x)}$, whereas the equivalence of the stationarity is further investigated in \cite{Bot-Dao-Li:2023SIAM-OPT}. This transform can be trivially adapted to the following fractional minimization problem
\begin{equation}\label{1219s1}
	\min \left\{ - \frac{f(x)}{g(x)} + h(x):x\in\Omega\cap\mathcal{C} \right\}.
\end{equation}
The resulting minimization problem is in the form of
\begin{equation}\label{1219s2}
	\min \left\{-2 y \sqrt{f(x)}+y^2 g(x) + h(x): x \in \Omega\cap \mathcal{C} ,~ y \in \mathbb{R}\right\} .
\end{equation}
and the equivalence between \eqref{1219s1} and \eqref{1219s2} at the levels of global optimality and stationarity follows immediately.
However, the main fractional minimization problem \eqref{problem:root} considered in this paper is generally independent of \eqref{1219s1} due to the non-negativity requirements imposed on $f$ and $g$. 
Although one may further reformulate \eqref{problem:root} into
\begin{equation}\label{1219s4}
	\min_{x \in \Omega \cap \mathcal{C}} \; \max_{y \in \mathbb{R}} \; 2 y \sqrt{f(x)} - y^2 g(x) + h(x),
\end{equation}
by exploiting again the identity $f(x)/g(x) = \max_{y\in\mathbb{R}} 2y \sqrt{f(x)}-y^2 g(x)$,
we emphasize that this formulation \textit{goes beyond} the scope of the quadratic transform, as it is a min-max problem. The equivalence between the stationarity of \eqref{problem:root} and that of \eqref{1219s4} is in general unknown and requires further investigation.
In addition, as in the treatment of problem \eqref{1219s2}, addressing the min-max problem \eqref{1219s4} typically requires $\sqrt{f}$ to be locally Lipschitz continuous, or even convex in some settings.
These requirements are substantially stronger than assuming local Lipschitz continuity of $f$ itself, which we adopt for the problem \eqref{problem:root}. For example, the function $f(x)=\|x\|_1$ is convex and locally Lipschitz continuous, whereas $\sqrt{f(x)}=\sqrt{\|x\|_1}$ is neither convex nor locally Lipschitz continuous. 

We also note that some algorithms have been specifically designed for solving the scale-variant sparse signal recovery model, which are based on neither the Dinkelbach's method nor the quadratic transform. The alternating direction method of multipliers (ADMM) has been applied to the noiseless $L_{1}/L_{2}$ signal recovery model and the model \eqref{problem: L1dL2} with $\mu = 0$  via various variable-splitting schemes \cite{Rahimi-Wang-Dong-Lou:2019SIAM-SC,Wang-Tao-Nagy-Lou:2021SIAM-ImageScience}. However, these ADMM-type algorithms rely on the structure of the $L_{1}/L_{2}$ function and can not be extended to deal with problem \eqref{problem:root} which admits a more general setting for $f$ and $g$. For example, they can not be generalized to solve model \eqref{problem: L1dSK} due to the nonsmoothness of $\|\cdot\|_{(K)}$. Very recently, the gradient descent flow algorithm (GDFA) is proposed for models \eqref{problem: L1dL2} and \eqref{problem: L1dSK} with $\mu=0$ \cite{Wang-Aujol-Gilboa-Lou:2024IPI}. Although the efficiency of GDFA is demonstrated in the numerical experiments of \cite{Wang-Aujol-Gilboa-Lou:2024IPI}, so far its convergence remains unknown to the best of our knowledge.

Motivated by the above analysis, we aim at designing efficient numerical methods with guaranteed convergence for solving problem \eqref{problem:root} by exploiting its intrinsic structure.
Specifically, the contributions of this paper are summarized as follows. 

First, we present a novel nonfractional min-max reformulation for problem \eqref{problem:root} 
\begin{equation}\label{problem: fractional min-max} 
	\min_{x \in \mathbb{R}^n} \max_{c \in \mathbb{R}}\; \wF(x,c) := 2c f(x)-c^2 f(x) g(x)+h(x) + \iota_{\Omega\cap \mathcal{C}}(x),
\end{equation}
where $\iota_{\Omega\cap \mathcal{C}}:\mathbb{R}^n \to \{0,+\infty\}$ is the indicator function on $\Omega\cap \mathcal{C}$.
The reformulation \eqref{problem: fractional min-max} is equivalent to \eqref{problem:root} in the sense that $x^{\star} \in \mathbb{R}^{n}$ is a global (local) minimizer of problem \eqref{problem:root} if and only if $(x^{\star}, c_{\star})$ is a global (local) min-max point of problem \eqref{problem: fractional min-max} for some $c_{\star} \in \mathbb{R}$. Moreover, we prove that $x^{\star}$ is a (Fr{\'e}chet) stationary point of problem \eqref{problem:root} if and only if $(x^{\star}, c_{\star})$ is a stationary point of problem \eqref{problem: fractional min-max} for some $c_{\star} \in \mathbb{R}$. To the best of our knowledge, problem \eqref{problem: fractional min-max} is the first equivalent min-max reformulation for fractional programs of the type \eqref{problem:root}.

Second, we propose an alternating maximization proximal descent algorithm (AMPDA) for solving the min-max problem \eqref{problem: fractional min-max}. 
Motivated by the min-max reformulation \eqref{problem: fractional min-max}, AMPDA adopts a natural alternating strategy: at the $k$-th iteration, fixing $x^k$, one maximizes $\wF(x^k,c)$ with respect to $c$, which admits a closed-form update $c_k = 1/g(x^k)$; fixing $c_k$, one then minimizes the possibly nonsmooth and nonconvex function $\wF(x,c_k)$ with respect to $x$. 
To handle the $x$-subproblem, we construct a majorization surrogate of $\wF(x,c_k)$ in a neighborhood of the current iterate $x^k$. 
Based on this surrogate, the $x$-update at the $k$-th iteration can be efficiently computed via the proximity operator associated with $f_{\mathcal{C}}$. 
For many practical applications, including the signal recovery models \eqref{problem: L1dL2} and \eqref{problem: L1dSK}, this proximity operator can be evaluated efficiently. 
The proximal stepsize is determined by a line-search procedure, whose stopping criterion is stronger than the standard sufficient descent condition for the objective in \eqref{problem:root}, but is crucial for establishing the convergence of the entire solution sequence generated by the proposed algorithm.
We emphasize that, despite the recent advances in nonsmooth min-max optimization (see, e.g., \cite{Bian-Chen:2024MP,Dai-Wang-Zhang:2024SIAM-OPT,He-Zhang-Xu:2024JGO,Kong-Monteiro:2021SIAM-OPT,Li-Zhu-So:2022NeurIPS-OPT,Lu-Tsaknakis-Hong-Chen:2020TSP}), existing methods are not suitable for solving the min-max problem \eqref{problem: fractional min-max}. 
The main difficulty arises from the fact that, for a fixed maximization variable $c \in \mathbb{R}$, the objective function involves a nonsmooth and nonconvex product term $-c^2 f(x) g(x)$ with respect to the minimization variable $x$. 
The construction of the surrogate function in the $x$-update of the proposed AMPDA is therefore not generic but carefully tailored to the specific structure of the term $-c^2f(x)g(x)$ in the objective of \eqref{problem: fractional min-max}, which makes the development of AMPDA highly nontrivial.

Third, we establish the convergence analysis of the proposed AMPDA. 
We show that any accumulation point of the solution sequence generated by AMPDA is a critical point of problem \eqref{problem:root} under a mild assumption. 
Moreover, we prove that AMPDA can find an $\epsilon$-critical point of problem \eqref{problem:root} within $\mathcal{O}(\epsilon^{-2})$ iterations. 
By further assuming that an auxiliary function is a proper closed KL function, we establish the convergence of the entire solution sequence generated by AMPDA. 
It is worth noting that our whole-sequence convergence analysis for AMPDA does not require any additional assumption on the denominator $g$, whereas existing nonsmooth fractional optimization methods typically assume that either the conjugate of $g$ is continuous on its domain or $g$ is locally Lipschitz differentiable on $\Omega$; see, for example, \cite{Bot-Dao-Li:2021MathematicsofOperationsResearch,Li-Shen-Zhang-Zhou:2022ACHA}.

The rest of this paper is organized as follows. In Section \ref{section:Notation_and_preliminaries}, we introduce notation and present some preliminary materials. In Section \ref{section: Connections between primal and primal-dual problems}, we study the relationships between problem \eqref{problem:root} and its min-max reformulation \eqref{problem: fractional min-max}. We propose AMPDA and establish its subsequential convergence in Section \ref{section: AMDPA}. The sequential convergence of AMPDA is further investigated in Section \ref{section: Global convergence analysis}. Finally, we present some numerical results of AMPDA in Section \ref{section: Numerical experiments}.


\section{Notation and preliminaries}\label{section:Notation_and_preliminaries}

We begin with our preferred notations. We denote the Euclidean space of dimension $n$ and the set of nonnegative integers by $\mathbb{R}^n$ and $\mathbb{N}$. The $\ell_1$-norm, the $\ell_2$-norm and the inner-product of $\mathbb{R}^n$ are denoted by $\|\cdot\|_1,~\|\cdot\|_2$ and $\innerP{\cdot}{\cdot}$, respectively. We use $\mathcal{B}(x,\delta)$ to denote an open ball centered at $x\in\mathbb{R}^n$ with radius $\delta>0$, i.e., $\mathcal{B}(x,\delta):=\{z\in\mathbb{R}^n:\|x-z\|_2 < \delta \} $. The Cartesian product of the sets $\mathcal{A}_1$ and $\mathcal{A}_2$ is denoted by $\mathcal{A}_1 \times \mathcal{A}_2$. The distance from a vector $x\in\mathbb{R}^n$ to a set $\mathcal{A}\subseteq \mathbb{R}^n$ is denoted by $\dist(x,\mathcal{A}) := \inf\{\|x-y\|_2:y\in\mathcal{A} \}$, and we adopt $\dist(x,\emptyset)=+\infty$ for convention. The indicator function on a nonempty set $\mathcal{A}\subseteq\mathbb{R}^n$ is defined by
\begin{equation*}
	\iota_{\mathcal{A}}(x)=
	\begin{cases}
		0, & \text{if $x\in\mathcal{A}$,}\\
		+\infty, & \text{else.}
	\end{cases}
\end{equation*}

An extended-real-value function $\varphi:\mathbb{R}^n\to (-\infty,+\infty]$ is said to be proper if its domain $\dom(\varphi) := \{x\in\mathbb{R}^n: \varphi(x)<+\infty \}$ is nonempty. A proper function $\varphi$ is said to be closed if $\varphi$ is lower semi-continuous on $\mathbb{R}^n$.
A function $\varphi:\mathbb{R}^n\to(-\infty,+\infty]$ is said to be locally Lipschitz continuous around $x\in\dom(\varphi)$ with a modulus $L_x>0$ (resp., relative to $\dom(\varphi)$), if there exists $\delta_x>0$ such that  $|\varphi(u)-\varphi(v)|_2\leq L_x\|u-v\|_2$ holds, for all $u,v\in\mathcal{B}(x,\delta_x)$ (resp., for all $u,v\in\mathcal{B}(x,\delta_x)\cap\dom(\varphi)$).
A function $\varphi:\mathbb{R}^n\to\mathbb{R}$ is said to be locally Lipschitz differentiable around $x\in\dom(\varphi)$ with a modulus $\tilde{L}_x>0$, if there exists $\delta_x>0$, such that 
$\varphi$ is continuously differentiable on $\mathcal{B}(x,\delta_x)$ and
$\|\nabla\varphi(u)-\nabla\varphi(v)\|_2\leq \tilde{L}_x \|u-v\|_2$ holds for all $u,v\in\mathcal{B}(x,\delta_x)$. For a proper closed function $\varphi: \mathbb{R}^{n} \to(-\infty,+\infty]$, the proximal operator of $\varphi$ at $x \in \mathbb{R}^{n}$ is denoted by $\prox_{\varphi}(x)$ and defined as
\begin{equation}
	\prox_{\varphi}(x) := 
	\arg\min\left\{ \varphi(z)+\frac{1}{2}\|x-z\|^2_2: z\in\mathbb{R}^n \right\}.
\end{equation}
The proximal operator $\prox_{\varphi}$ takes an $x\in\mathbb{R}^n$ and maps it into a closed subset of $\mathbb{R}^n$, which may be empty, a singleton or a subset with multiple vectors. If for a given $x\in\mathbb{R}^n$, the function $z\to \varphi(z)+\frac{1}{2}\|x-z\|^2_2 $ is coercive, then $\prox_{\varphi}(x)$ is nonempty (see, for example, \cite[Theorem 6.4]{Beck-Amir:2017SIAM}). Moreover, if $\varphi$ is convex, then $\prox_{\varphi}(x)$ is a singleton for all $x\in\mathbb{R}^n$. In particular, for problem \eqref{problem:root}, the proximal operator  $\prox_{f_{\mathcal{C}}}(x)$ is nonempty for all $x\in\mathbb{R}^n$ due to $f_{\mathcal{C}}(x)\geq 0$.

\subsection{Preliminaries on nonsmooth analysis}\label{subsection: Generalized subdifferential}
We first review some preliminaries on subdifferentials of nonconvex functions \cite{PangJongShi-CuiYing:2021ModernNonconvexNondifferentiableOptimization,Boris:2006Variational_analysis,Rockafellar2004Variational}.
For a proper function $\varphi$, the Fr{\'e}chet subdifferential at $x\in\dom(\varphi)$ is defined by
\begin{equation*}
	\widehat{\partial} \varphi (x):=\left\{ y\in\mathbb{R}^n:\mathop{\lim\inf}\limits_{\substack{z\to x\\z\neq x}}\;
	\frac{\varphi(z)-\varphi(x)-\innerP{y}{z-x}}{\|z-x\|_2}\geq 0
	\right\},
\end{equation*}
and limiting subdifferential at $x\in\dom(\varphi)$ is defined by
\begin{multline*}
	{\partial}\varphi(x) := \{ y\in\mathbb{R}^n:\exists x^k\to x,~\varphi(x^k)\to\varphi(x),~\\ y^k \to y \text{ with $y^k\in\widehat{\partial}\varphi(x^k)$ for each $k$} \}.
\end{multline*}
We set $\widehat{\partial} \varphi(x)=\partial \varphi(x) = \emptyset$ by convention when $x\notin\dom(\varphi)$, and define $\dom(\partial \varphi):= \{x:\partial \varphi(x)\neq \emptyset \}$.
It is clear that $\widehat{\partial}(\alpha \varphi)(x) = \alpha \widehat{\partial}\varphi(x)$ for all $x\in\dom(\varphi)$ and $\alpha >0$.
Moreover, $\widehat{\partial}\varphi(x)\subseteq {\partial}\varphi(x)$ for all $x\in\dom(\varphi)$. If $\varphi$ is differentiable at $x$, then  $\widehat{\partial} \varphi (x) = \{\nabla\varphi(x) \}$, and if $\varphi$ is continuously differentiable around $x$, one also has $\partial\varphi(x) =\{\nabla\varphi(x) \}$ (\cite[Exercise 8.8(b)]{Rockafellar2004Variational}). For a convex $\varphi$, the Fr{\'e}chet and limiting subdifferentials coincide with the classical subdifferential of a convex function at each $x\in\dom(\varphi)$ (\cite[Proposition 8.12]{Rockafellar2004Variational}); namely,
\begin{equation*}
	\widehat{\partial} \varphi(x) = \partial\varphi(x) = \{y\in\mathbb{R}^n:\varphi(z)-\varphi(x)-\innerP{y}{z-x}\geq0, \text{~for all~} z\in\mathbb{R}^n \}.
\end{equation*}
If  $\varphi(x) = \varphi_1(x_1)+\varphi_2(x_2)$ for proper closed functions $\varphi_1:\mathbb{R}^{n_1} \to (-\infty,+\infty]$ and $\varphi_2:\mathbb{R}^{n_2} \to (-\infty,+\infty]$, where $x\in\mathbb{R}^{n}$ is expressed as $(x_1,x_2)$,
then it holds that $\widehat{\partial}\varphi(x) = \widehat{\partial}\varphi_1(x_1)\times \widehat{\partial}\varphi_2(x_2)$ (\cite[Proposition 10.5]{Rockafellar2004Variational}). For proper functions $\varphi_1,~\varphi_2:\mathbb{R}^n\to(-\infty,+\infty]$, there holds $\widehat{\partial}\varphi_1(x) + \widehat{\partial}\varphi_2(x)\subseteq\widehat{\partial}(\varphi_1+\varphi_2)(x)$ (\cite[Corollary 10.9]{Rockafellar2004Variational}), and equality occurs when $\varphi_1$ or $\varphi_2$ is differentiable at $x$ (\cite[Section 8.8 Exercise]{Rockafellar2004Variational}). For a function $\varphi:\mathbb{R}^n \to (-\infty,+\infty]$, the generalized \textit{Fermat's Rule} states that $0\in\widehat{\partial} \varphi(x^{\star})$ if $x^{\star}\in\dom(\varphi)$ is a local minimizer of $\varphi$ (\cite[Theorem 10.1]{Rockafellar2004Variational}).

We finally review some properties of real-valued convex functions. Let $\varphi:\mathbb{R}^n \to \mathbb{R}$ be convex. Then $\varphi$ is continuous on $\mathbb{R}^n$ \cite[Theorem 2.35]{Rockafellar2004Variational}. In view of this, one can verify that $\partial \varphi$ is outer semicontinuous, that is, $\bar{y}\in\partial\varphi (\bar{x})$ holds if there exists some $x^k\to \bar{x}$ and $y^k\to\bar{y}$ with $y^k\in\partial\varphi(x^k)$ for each $k\in\mathbb{N}$. For any bounded $\mathcal{S}\subseteq \mathbb{R}^n$, the $\varphi$ is Lipschitz continuous on $\mathcal{S}$, and $\cup_{x\in\mathcal{S}} \partial \varphi(x)$ is nonempty and bounded (\cite[Proposition 5.4.2]{Bertsekas:Convexoptimizationtheory}).
The convex conjugate function $\varphi^*:\mathbb{R}^n\to[-\varphi(0),+\infty]$ is defined at $y\in\mathbb{R}^n$ as $\varphi^*(y):= \sup\{\innerP{x}{y}-\varphi(x):x\in\mathbb{R}^n \}$. The conjugate $\varphi^*$ is a proper closed convex function \cite[Theorem 12.2]{Rockafellar:70}. Furthermore, for any $x,y\in\mathbb{R}^n$, the following three statements are equivalent (\cite[Proposition 11.3]{Rockafellar2004Variational}):  $\innerP{x}{y} = \varphi(x)+\varphi^*(y) \Leftrightarrow y\in\partial \varphi(x) \Leftrightarrow x\in\partial \varphi^*(y)$. Therefore, it is clear that the \textit{Fenchel-Young Inequality}, i.e., $\varphi(x)+\varphi^*(y)\geq \innerP{x}{y}$, holds for all $x,y\in\mathbb{R}^n$, and it becomes an equality when $y\in\partial \varphi(x)$. In addition, from the \textit{Fenchel-Young Inequality}, we know that $y\in\dom(\varphi^*)$ if $y\in\partial \varphi(x)$ for some $x\in\mathbb{R}^n$.

\subsection{KL property and semialgebra} \label{section: KL property and semialgebra} 
We now review the KL property, which has been used extensively in the convergence analysis of various first order methods.
\begin{definition}[KL property \cite{Attouch-bolt-redont-soubeyran:2010,Attouch-Bolte:MP:2013}]\label{Def:KL_property}
	A proper function $\varphi:\mathbb{R}^n\to(-\infty, +\infty]$ is said to satisfy the KL property at $x\in\mathrm{dom}(\partial \varphi)$ if there exist $\epsilon\in(0,+\infty]$, $\delta>0$ and a continuous concave function $\phi:[0,\epsilon) \to [0,+\infty)$ such that:
	\begin{enumerate}[\upshape(\romannumeral 1)]
		\item $\phi(0)=0$;
		\item $\phi$ is continuously differentiable on $(0,\epsilon)$ with $\phi'>0$;
		\item For any $z\in \mathcal{B}(x,\delta)$ satisfying $\varphi(x)<\varphi(z)<\varphi(x)+\epsilon$, it holds that $\phi'(\varphi(z)-\varphi(x))\mathrm{~dist}(0,\partial\varphi(z)) \geq 1$.
	\end{enumerate}
\end{definition}
A proper function $\varphi:\mathbb{R}^n\to (-\infty, +\infty]$ is called a KL function, if it satisfies the KL property at each point in $\dom(\partial \varphi)$. 
The following proposition provides a framework for full-sequential convergence analysis using the KL property. 
\begin{proposition}\label{ppsition:3.4-3} \cite[Proposition 2.7]{Li-Shen-Zhang-Zhou:2022ACHA} \cite[Theorem 4.3]{Sabach-Teboulle-Voldman:2018PAFA} 
	Let $\Psi:\mathbb{R}^n\times\mathbb{R}^m\to(-\infty, +\infty]$ be proper lower semicontinuous. A bounded sequence $\{(u^k,v^k)\in\mathbb{R}^n\times\mathbb{R}^m:k\in\mathbb{N} \}$ satisfies the following three conditions:
	\begin{enumerate}[{\upshape(\romannumeral 1)}]
		\item (Sufficient descent condition). There exists $a>0$, such that
		\begin{equation*}
			\Psi(u^{k+1},v^{k+1})+a\|u^{k+1}-u^k\|_2^2 \leq \Psi(u^k,v^k), \text{~~~~~for all $k\in\mathbb{N}$;}
		\end{equation*}  
		\item (Relative error condition). There exists $b>0$, such that
		\begin{equation*}
			\dist(0,\partial \Psi(u^{k+1},v^{k+1}))\leq b \|u^{k+1}-u^k\|_2, \text{~~~~~for all $k\in\mathbb{N}$;}
		\end{equation*}
		\item (Continuity condition).\footnote{In \cite[Theorem 4.3]{Sabach-Teboulle-Voldman:2018PAFA}, this condition is replaced by $\lim\sup_{k\in\mathcal{K}\subseteq \mathbb{N}} \Psi(u^k,v^k) \leq \Psi(u^{\star},v^{\star}) $, where $(u^{\star},v^{\star})$ is a limit point of a subsequence $\{(u^{k},v^{k}) \}_{k\in\mathcal{K}}$.} The limit	$\Psi_{\infty} := \lim_{k\to\infty}\Psi(u^k,v^k)$ exists and $\Psi\equiv \Psi_{\infty}$ holds on $\Upsilon$, where $\Upsilon$ denotes the set of accumulation points of $\{(u^k,v^k):k\in\mathbb{N} \}$.
	\end{enumerate}
	If $\Psi$ satisfies the KL property at each point of $\Upsilon$, then we have $\sum_{k=0}^{\infty}\|u^{k+1}-u^k\|_2<+\infty$, $\lim_{k\to\infty}u^k=u^{\star}$ and $0\in\partial \Psi(u^{\star},v^{\star})$ for some $(u^{\star},v^{\star})\in\Upsilon$.
\end{proposition}

Proper closed semialgebraic functions constitute a broad and important class of KL functions.
A function $\varphi:\mathbb{R}^n\to (-\infty,+\infty]$ is called semialgebraic if its graph $\mathrm{Graph}(\varphi):=\{(x,s)\in\mathbb{R}^n\times\mathbb{R}:s=\varphi(x)\}$ is a semialgebraic subset of $\mathbb{R}^{n+1}$, that is, there exist a finite many real polynomials $G_{ij}$, $H_{ij}:\mathbb{R}^{n+1}\to\mathbb{R}$ such that
\begin{equation*}
	\Graph(\varphi)=\bigcup^p_{j=1}\,\bigcap^q_{i=1}\,\{z\in\mathbb{R}^{n+1}:G_{ij}(z)=0,~H_{ij}(z)<0 \}.
\end{equation*}
It is clear that any real polynomial functions is semialgebraic. 
More generally, the class of semialgebraic functions is stable under finite sums, products, and composition, and the convex conjugate of a semialgebraic function is semialgebraic.
We refer the reader to \cite{Attouch-bolt-redont-soubeyran:2010,Bolte-Sabach-Teboulle:MP:2014} for more details on semialgebraic functions and the KL property.


\subsection{Min-max problem}
Consider the following min-max problem:
\begin{equation}\label{problem: min-max general}
	\min_{x\in \mathcal{F}_X} \max_{y\in\mathcal{F}_Y}~ \psi(x,y),
\end{equation}
where $\mathcal{F}_X\subseteq \mathbb{R}^n$ and $\mathcal{F}_Y\subseteq \mathbb{R}^m$ are nonempty, and the objective $\psi:\mathbb{R}^n\times \mathbb{R}^m \to\mathbb{R}$ is continuous. We assume that $\mathop{\arg\max}_{y\in\mathcal{F}_Y} \psi(x,y) \neq \emptyset$ for all $x\in\mathcal{F}_X$, and we define an envelop function $\varphi:\mathcal{F}_{X}\to\mathbb{R} $ as
\begin{equation*}
	\varphi(x) := \max_{y\in\mathcal{F}_Y}~ \psi(x,y).
\end{equation*}
Problem \eqref{problem: min-max general} is called nonconvex-concave if for a fixed $x\in\mathcal{F}_X$, $\psi(x,\cdot)$ is concave, and for a fixed $y\in\mathcal{F}_Y$, $\psi(\cdot,y)$ is not convex. A point  $(x^{\star},y^{\star})\in\mathcal{F}_X \times \mathcal{F}_Y$ is said to be a saddle point of problem \eqref{problem: min-max general} if $\psi(x^{\star},y)\leq \psi(x^{\star},y^{\star})\leq  \psi(x,y^{\star})$ holds for all $(x,y)\in\mathcal{F}_X \times \mathcal{F}_Y$. The concept of saddle points is usually used to describe the optimality of the convex-concave min-max problems. Unfortunately, it is well-known that such points may not exist for nonconvex-concave min-max problems. Due to this, the notions of global and local min-max points are introduced in \cite[Definition 9]{Jin-Netrapalli-Jordan:2020ICML} and \cite[Definition 14]{Jin-Netrapalli-Jordan:2020ICML}, respectively. Here we consider equivalent definitions of these two concepts given in  \cite[Remark 10 and Lemma 16]{Jin-Netrapalli-Jordan:2020ICML} as follows.

\begin{definition}[global min-max point] \label{def: MM-global MinMax}
	\cite[Remark 10]{Jin-Netrapalli-Jordan:2020ICML}
	Let $(x^{\star},y^{\star})\in\mathcal{F}_X \times \mathcal{F}_Y$. We say that $(x^{\star},y^{\star})$
	is a global min-max point of problem \eqref{problem: min-max general}, if $y^{\star}$ is a global maximizer of $\psi(x^{\star},\cdot)$ and $x^{\star}$ is a global minimizer of $\varphi$, i.e.,
	\begin{equation}\label{eq: definition of global min-max point}
		\begin{cases}
			\varphi(x^{\star})\leq \varphi(x),
			&\text{for all $x\in \mathcal{F}_X$,}\\
			\psi(x^{\star},y^{\star}) \geq \psi(x^{\star},y) , &\text{for all $y\in \mathcal{F}_Y$.}
		\end{cases}
	\end{equation}
\end{definition}

\begin{definition}[local min-max point] 
	\label{def: local min-max point}
	\cite[Lemma 16]{Jin-Netrapalli-Jordan:2020ICML}
	We say that $(x^{\star},y^{\star})\in\mathcal{F}_X \times \mathcal{F}_Y$ is a local min-max point of problem \eqref{problem: min-max general}, if $y^{\star}$ is a local maximizer of $\psi(x^{\star},\cdot)$ relative to $\mathcal{F}_Y$, and there exists an $\epsilon_0>0$ such that $x^{\star}$ is always a local minimizer of $\varphi_{\epsilon}$ for all $\epsilon\in(0,\epsilon_0]$ relative to $\mathcal{F}_X$, where the function $\varphi_{\epsilon}$ is defined as $\varphi_{\epsilon}(x):=\max\{ \psi(x,y):y\in\mathbb{R}^m, \|y-y^{\star}\|_2\leq \epsilon \}$.
\end{definition}

Due to the lower semicontinuity of $\varphi$, problem \eqref{problem: min-max general} always has a global min-max point if $\mathcal{F}_X$ is compact. A saddle point of problem \eqref{problem: min-max general} is also a local min-max point. For more discussions on their relationships, we refer the readers to  \cite{Jiang-Chen:2023SIAM-OPT,Jin-Netrapalli-Jordan:2020ICML}.

\section{The connections between problem \eqref{problem:root} and the min-max reformulation \eqref{problem: fractional min-max}}\label{section: Connections between primal and primal-dual problems}

In this section, we investigate the connections between global and local optimal solutions as well as stationary points of problems \eqref{problem:root} and \eqref{problem: fractional min-max}. To this end, we first introduce the extended objective $F: \mathbb{R}^{n} \to (-\infty,+\infty]$ for problem \eqref{problem:root}:
\begin{equation}\label{eq: def F}
	F(x)  := \begin{cases}
		\frac{f(x)}{g(x)}+h(x), & \text{if $x\in\Omega\cap\mathcal{C}$},\\
		+\infty, & \text{else.}
	\end{cases}
\end{equation}
Obviously, problem \eqref{problem:root} can be rewritten as $\min\{F(x):x\in\mathbb{R}^n \}$. 
Recall that $\wF$ denotes the objective of the min-max problem \eqref{problem: fractional min-max}, as defined in Section \ref{section:introduction}. 
By a direct computation, we obtain that for all $(x,c)\in(\Omega\cap\mathcal{C})\times \mathbb{R}$, 
\begin{equation}\label{eq: relation between F(x;c) and F(x)}
	F(x) - \wF(x,c)  = f(x)g(x)\left( \frac{1}{g(x)}-c \right)^2.
\end{equation}
In view of \eqref{eq: relation between F(x;c) and F(x)} as well as the non-negativity of $f$ and $g$, we know that for all $x\in\Omega\cap\mathcal{C}$, 
\begin{equation}\label{eq: F = max_c wF}
	F(x)  = \max_{c\in\mathbb{R}} \wF(x,c),
\end{equation}
where the maximum is attained at some $c_x\in\mathbb{R}$ satisfying $f(x)c_x=f(x)/g(x) $. The following proposition shows the relationship between the global minimizers of problem \eqref{problem:root} and the global min-max points of problem \eqref{problem: fractional min-max}.

\begin{proposition}
	If $x^{\star}$ is a global minimizer of problem \eqref{problem:root}, then $(x^{\star},1/g(x^{\star}))$ is a global min-max point of problem \eqref{problem: fractional min-max}. 
	Conversely, if $(x^{\star},c_{\star})$ is a global min-max point of problem \eqref{problem: fractional min-max}, then $x^{\star}$ is a global minimizer of problem \eqref{problem:root}.
\end{proposition}

\begin{proof}
	If $x^{\star}$ is a global minimizer of \eqref{problem:root}, then we have $x^{\star} \in \arg \min \{F(x): x \in \Omega \cap \mathcal{C}\}$. This, together with \eqref{eq: F = max_c wF}, leads to
	\begin{equation}\label{eq: z08251506}
		x^{\star} \in \arg \min \left\{\max _{c \in \mathbb{R}} \wF(x, c): x \in \Omega \cap \mathcal{C} \right\} .
	\end{equation}
	In view of \eqref{eq: F = max_c wF}, we derive that $1 / g(x^{\star}) \in \arg\max\{\wF(x^{\star},c):c\in\mathbb{R} \}$, which along with \eqref{eq: z08251506} indicates that $(x^{\star}, 1 / g(x^{\star}))$ is a global min-max point of \eqref{problem: fractional min-max}, according to Definition \ref{def: MM-global MinMax}.
	
	Conversely, if $(x^{\star}, c_{\star})$ is a global min-max point of \eqref{problem: fractional min-max}, then we derive from Definition \ref{def: MM-global MinMax} that $x^{\star}$ satisfies \eqref{eq: z08251506} . Thus, we conclude that $x^{\star} \in \arg\min\{F(x): x \in \Omega \cap \mathcal{C}\}$ from \eqref{eq: z08251506} and \eqref{eq: F = max_c wF}.  This completes the proof.
	\hfill $\square$ \end{proof}

We next establish the connection between the local minimizers of problem \eqref{problem:root} and the local min-max points of problem \eqref{problem: fractional min-max}.

\begin{proposition}
	If $x^{\star}$ is a local minimizer of problem  \eqref{problem:root}, then $(x^{\star},1/g(x^{\star}))$ is a local min-max point of problem \eqref{problem: fractional min-max}. 
	Conversely, if $(x^{\star},c_{\star})$ is a local min-max point of problem \eqref{problem: fractional min-max}, then $x^{\star}$ is a local minimizer of problem \eqref{problem:root}.
\end{proposition}

\begin{proof}
	We first prove that the statement that $(x^{\star}, 1 / g(x^{\star}))$ is a local min-max point of problem \eqref{problem: fractional min-max} follows from the $x^{\star}$ being a local minimizer of problem \eqref{problem:root}. Since \eqref{eq: F = max_c wF} indicates that $1 / g(x^{\star})$ is a global maximizer of $\wF(x^{\star}, \cdot)$, it suffices to show that, for an arbitrarily fixed $\epsilon>0$, the $x^{\star}$ is a local minimizer of
	\begin{equation}\label{eq: def varphi_eps}
		\varphi_{\epsilon}(x):=\max\left\{ \wF(x, c): c\in\mathbb{R},~|c-1 / g(x^{\star})| \leq \epsilon\right\},
	\end{equation}
	relative to $\Omega\cap\mathcal{C}$ according to Definition \ref{def: local min-max point}. Let this $\epsilon>0$ be fixed. Firstly, due to the continuity of $1 / g$ around $x^{\star}\in\Omega$, there exists $\delta_{\epsilon}>0$ such that $|1 / g(x)-1 / g(x^{\star})| \leq \epsilon$ for all $x \in \mathcal{B}(x^{\star}, \delta_{\epsilon})$. In view of this, we deduce from the definition \eqref{eq: def varphi_eps} that
	\begin{equation}\label{eq: z08251611}
		\wF(x, 1 / g(x)) \leq \varphi_{\varepsilon}(x) \text {,~~ for all } x \in \mathcal{B}(x^{\star}, \delta_{\epsilon})\cap(\Omega \cap \mathcal{C}) \text {. }
	\end{equation}
	Secondly, since $x^{\star}$ is a local minimizer of \eqref{problem:root}, there exists $\delta_{\epsilon}' \in(0, \delta_{\epsilon})$, such that
	\begin{equation}\label{eq: z08251610}
		F(x^{\star}) \leq F(x) \text {,~~ for all } x \in \mathcal{B}(x^{\star}, \delta_{\epsilon}') \cap(\Omega \cap \mathcal{C}) \text {. }
	\end{equation}
	Thus, it holds for all $x \in \mathcal{B}(x^{\star}, \delta_{\epsilon}') \cap(\Omega \cap \mathcal{C})$ that
	\begin{equation}\label{eq: z08251612}
		\varphi_{\epsilon}(x^{\star}) \stackrel{\StatementNumUpperCase{1}}{\leq} \max _{c \in \mathbb{R}}
		\wF(x^{\star}, c) \stackrel{\StatementNumUpperCase{2}}{=} F(x^{\star}) \stackrel{\StatementNumUpperCase{3}}{\leq} F(x) \stackrel{\StatementNumUpperCase{4}}{=} \widetilde{F}(x, 1 / g(x)) \stackrel{\StatementNumUpperCase{5}}{\leq} \varphi_{\varepsilon}(x),
	\end{equation}
	where $\StatementNumUpperCase{1}\sim\StatementNumUpperCase{5}$ follow from \eqref{eq: def varphi_eps}, \eqref{eq: F = max_c wF}, \eqref{eq: z08251610}, \eqref{eq: F = max_c wF} and \eqref{eq: z08251611}, respectively. According to \eqref{eq: z08251612}, we deduce that $x^{\star}$ is a local minimizer of $\varphi_{\epsilon}$ for an arbitrary $\epsilon>0$.
	
	Conversely, we let $(x^{\star}, c_{\star})$ be a local min-max point of \eqref{problem: fractional min-max}. Then, according to Definition \ref{def: local min-max point}, we know that $c_{\star}$ is a local maximizer of $\wF(x^{\star}, \cdot)$. Hence, we have $0=\nabla(\wF(x^{\star}, \cdot))(c_{\star})$, which leads to $f(x^{\star}) c_{\star}=f(x^{\star}) / g(x^{\star})$. 	
	In addition, also according to Definition \ref{def: local min-max point}, we know that $x^{\star}$ is a local minimizer of $\varphi_{\epsilon}$ with some $\epsilon>0$ relative to $\Omega \cap \mathcal{C}$. Thus, there exists some $\delta_{\epsilon}>0$ such that
	\begin{equation}\label{eq: z08251633}
		\varphi_{\epsilon}(x^{\star})\leq \varphi_{\epsilon}(x), \text{\quad for all $x\in\mathcal{B}(x^{\star},\delta_{\epsilon})\cap (\Omega\cap\mathcal{C})$.}
	\end{equation}
	For the left side of \eqref{eq: z08251633}, we have
	\begin{equation}\label{eq: z08251636}
		F(x^{\star})=\wF(x^{\star}, c_{\star}) \leq \varphi_{\epsilon}(x^{\star}),
	\end{equation}
	where the first equation follows from \eqref{eq: F = max_c wF} and $f(x^{\star}) c_{\star}=f(x^{\star}) / g(x^{\star})$.	As to the right side of \eqref{eq: z08251633}, it is clear that
	\begin{equation}\label{eq: z08251637}
		\varphi_{\epsilon}(x) \leq \max \left\{\wF(x, c): c \in \mathbb{R}\right\}=F(x),
		\quad \text{for all $x \in \Omega \cap \mathcal{C}$.}
	\end{equation}
	Together with \eqref{eq: z08251636} and \eqref{eq: z08251637}, we deduce from \eqref{eq: z08251633} that $F(x^{\star}) \leq F(x)$ for any $x \in \mathcal{B}(x^{\star}, \delta_{\epsilon}) \cap(\Omega \cap \mathcal{C})$, which yields that $x^{\star}$ is a local minimizer of \eqref{problem:root}. 
	\hfill $\square$ \end{proof}

Now we study the relationship between stationary points of problems \eqref{problem:root} and \eqref{problem: fractional min-max}. To this end, for $(x,c)\in(\Omega\cap\mathcal{C})\times\mathbb{R} $, we let $\widehat{\partial}_x \wF(x,c)$ and $\nabla_c \wF(x,c)$ denote the partial Fr{\'e}chet subdifferential of $\wF$ with respect to $x$ and the partial derivative of $\wF$ with respect to $c$, respectively.

\begin{proposition}\label{Prop: relation between the stationary points of F and wF}
	If $0\in \widehat{\partial} F(x^{\star})$, then for $c_{\star} = 1/g(x^{\star})$ it holds that
	\begin{equation}\label{eq: z08172057}
		\left\{
		\begin{aligned}
			0\in \widehat{\partial}_x \, \wF (x^{\star},c_{\star}),\\
			0= \nabla_c \wF (x^{\star},c_{\star}).
		\end{aligned}\right.
	\end{equation}
	Conversely, if $(x^{\star},c_{\star})$ satisfies \eqref{eq: z08172057}, then $0\in \widehat{\partial} F(x^{\star})$.
\end{proposition}

\begin{proof}
	Let $(x^{\star},c_{\star})\in(\Omega\cap\mathcal{C})\times\mathbb{R}$.
	Due to the continuity of $g$ around $x^{\star} \in \Omega$, there exists $\delta>0$ such that $g(x)>0$ holds for all $x\in\mathcal{B}(x^{\star}, \delta)$. Then, we have
	\begin{equation}\label{eq: z08252005}
		\widehat{\partial}_x\wF (x^{\star}, c_{\star})
		=\widehat{\partial}\left(f/g+h+\iota_{\mathcal{C}}-\varphi\right)(x^{\star})=\widehat{\partial}(F-\varphi)(x^{\star}),
	\end{equation}
	where $\varphi(x):=\frac{f(x)}{g(x)}-2c_{\star}f(x) + c_{\star}^2f(x)g(x)$, and equally, we have
	\begin{equation}\label{eq: z08251951}
		\varphi(x)=f(x) g(x)\left(\frac{1}{g(x)}-c_{\star}\right)^{2}, \text {~~~ for all $ x \in \mathcal{B}(x^{\star}, \delta)$}.
	\end{equation}
	
	
	We first show that \eqref{eq: z08172057} holds if  $0 \in \widehat{\partial} F(x^{\star})$ and $c_{\star}=1 / g(x^{\star})$. Firstly, since $1 / g(x^{\star}) \in \arg\max\{ \wF(x^{\star},c): c \in \mathbb{R} \}$ (see \eqref{eq: F = max_c wF}) and $\wF(x^{\star}, \cdot)$ is differentiable, we derive that $0=\nabla_{c}\wF(x^{\star}, c_{\star})$ for the $c_{\star}=1 / g(x^{\star})$. Secondly, according to \eqref{eq: z08251951} and $c_{\star}=1 / g(x^{\star})$, we have $|\varphi(x)-\varphi(x^{\star})|=f(x)(g(x))^{-1}(g(x^{\star}))^{-2}|g(x)-g(x^{\star})|^{2}$. This implies that $\varphi$ is differentiable at $x^{\star}$ with $\nabla \varphi\left(x^{\star}\right)=0$, since the real-valued convex function $g$ is Lipschitz continuous on $\mathcal{B}(x^{\star}, \delta)$. Therefore, $0 \in \widehat{\partial} F(x^{\star})$ yields $0 \in \widehat{\partial}(F-\varphi)(x^{\star})$, which together with \eqref{eq: z08252005} implies that $0 \in \widehat{\partial}_{x}\wF(x^{\star}, c_{\star})$.
	
	Conversely, we suppose that \eqref{eq: z08172057} holds. Then $0=\nabla_{c}\wF(x^{\star}, c_{\star})$ implies that $f(x^{\star}) c_{\star}=f(x^{\star}) / g(x^{\star})$. Plugging this into \eqref{eq: z08251951} and noting that $f$ and $g$ are non-negative, we deduce that $\varphi(x^{\star})=0$ is the minimum of $\varphi$ on $ x \in \mathcal{B}(x^{\star}, \delta)$. Hence, by invoking the generalized \textit{Fermat's Rule}, we derive $0\in\widehat{\partial}\varphi(x^{\star})$. With the help of \cite[Theorem 3.1]{Mordukhovich:2006Frechet_subdifferential}, we have 
	\begin{equation}\label{eq: z10092147}
		\widehat{\partial}(F-\varphi)(x^{\star}) \subseteq \bigcap_{v^{\star}\in\widehat{\partial}\varphi(x^{\star})} (\widehat{\partial}F(x^{\star}) - v^{\star})\subseteq \widehat{\partial}F(x^{\star})-\{ 0\}.
	\end{equation}
	Together with the first relation of \eqref{eq: z08172057}, the equation \eqref{eq: z08252005}, and the relation  \eqref{eq: z10092147}, we finally conclude that $0 \in \widehat{\partial} F(x^{\star})$. This completes the proof.
	\hfill $\square$ \end{proof}

With the help of Proposition \ref{Prop: relation between the stationary points of F and wF}, we present a useful first-order necessary condition for problem \eqref{problem:root}.

\begin{corollary}\label{Corollary: critical point}
	If $0\in\widehat{\partial} F(x^{\star})$, then for $c_{\star} = 1/g(x^{\star})$  it holds that
	\begin{equation}\label{eq: def lifted stationary point}
		0\in \widehat{\partial}(c_{\star} f+\iota_{\mathcal{C}})\left(x^{\star}\right)-c_{\star}^{2} f(x^{\star}) \partial g\left(x^{\star}\right)+\nabla h_{1}\left(x^{\star}\right)-\partial h_{2}\left(x^{\star}\right).
	\end{equation}
	Moreover, if $g$ and $h_2$ are differentiable at $x^{\star}$, then $0\in\widehat{\partial} F(x^{\star})$ is equivalent to \eqref{eq: def lifted stationary point} with $c_{\star} = 1/g(x^{\star})$.
\end{corollary}

\begin{proof}
	Note that $\varphi_{2}:=2 c_{\star} f-c_{\star}^{2} f g-c_{\star} f+c_{\star}^{2} f(x^{\star}) g$ is differentiable at $x^{\star}$. Precisely, we have $\nabla \varphi_{2}(x^{\star})=0$, since
	\begin{equation*}
		\lim _{x \to x^{\star}} \frac{|\varphi_{2}(x)-\varphi_{2}(x^{\star})|}{\|x-x^{\star}\|_{2}}=\lim_{x \to x^{\star}} \frac{c_{\star}^{2}|f(x)-f(x^{\star})||g(x^{\star})-g(x)|}{\|x-x^{\star}\|_{2}}=0,
	\end{equation*}
	where the last equality holds, owing to $x^{\star}\in\dom(\widehat{\partial} F)\subseteq\Omega$,	
	and the local Lipschitz continuity of $f$ and $g$ on $\Omega$. Hence, for  $c_{\star} =1 / g(x^{\star})$, it holds that
	\begin{align}\label{eq: z08262052}
		& \widehat{\partial}_{x}\wF(x^{\star}, c_{\star})\overset{\StatementNumUpperCase{1}}{=}\widehat{\partial}(2 c_{\star} f-c_{\star}^{2} f g+h+\iota_{\mathcal{C}})(x^{\star}) \\
		& \overset{\StatementNumUpperCase{2}}{=} \widehat{\partial}(c_{\star} f-c_{\star}^{2} f(x^{\star}) g-h_{2}+\iota_{\mathcal{C}})(x^{\star})+\nabla \varphi_{2}(x^{\star})+\nabla h_{1}(x^{\star}) \notag \\
		& \overset{\StatementNumUpperCase{3}}{\subseteq} \widehat{\partial}(c_{\star} f+\iota_{\mathcal{C}})\left(x^{\star}\right)-c_{\star}^{2} f(x^{\star}) \partial g\left(x^{\star}\right)+\nabla h_{1}\left(x^{\star}\right)-\partial h_{2}\left(x^{\star}\right),\notag
	\end{align}
	where $\StatementNumUpperCase{1}$ follows from the definition of $\widetilde{F}$ (see \eqref{problem: fractional min-max}) and the open $\Omega$, $\StatementNumUpperCase{2}$ holds because $\varphi_{2}$ and $h_{1}$ are differentiable at $x^{\star}$, and $\StatementNumUpperCase{3}$ holds thanks to \cite[Theorem 3.1]{Mordukhovich:2006Frechet_subdifferential}, $c_{\star}^{2} f(x^{\star}) \geq 0$, and the real-valued convexity of the functions $g$ and $h_2$. Especially, $\StatementNumUpperCase{3}$ in \eqref{eq: z08262052} becomes an equality when $g$ and $h_2$ are differentiable at $x^{\star}$. Invoking \eqref{eq: z08262052} and Proposition \ref{Prop: relation between the stationary points of F and wF}, we then obtain the desired results.
	\hfill $\square$ \end{proof}

In general, it is impractical to find an $x^{\star}$ satisfying $0\in \widehat{\partial} F(x^{\star})$. Inspired by Corollary \ref{Corollary: critical point}, we introduce a weaker definition of a critical point of $F$ as follows, which is more amenable to algorithmic analysis and computation.

\begin{definition}\label{Def: lifted stationary point}
	We say that $x^{\star}\in\Omega $ is a critical point of $F$, if \eqref{eq: def lifted stationary point} holds with $c_{\star}= 1 / g(x^{\star})$.
\end{definition}

Using the proximity operator of $f_{\mathcal{C}} := f + \iota_{\mathcal{C}}$, one can derive that $\bar{x}$ is a critical point of $F$ whenever $\bar{x} \in \prox_{ \alpha \bar{c} f_{\mathcal{C}} }( \bar{x} - \alpha (\nabla h_1(\bar{x})-\bar{z} - \bar{c}^2 f(\bar{x}) \bar{y}))$, or, equivalently $$\dist\left(\bar{x},~ \prox_{ \alpha \bar{c} f_{\mathcal{C}} }\left( \bar{x} - \alpha (\nabla h_1(\bar{x})-\bar{z} - \bar{c}^2 f(\bar{x}) \bar{y})\right)\right) = 0$$ for some $\alpha >0$, $\bar{c} = 1/g(\bar{x})$, and $(\bar{y},\bar{z})\in \partial g(\bar{x})\times \partial h_2(\bar{x})$. Thus, for a given $\bar{x}$, the above distance naturally quantifies the violation of the critical point condition \eqref{eq: def lifted stationary point}. Due to this, below we define an $\epsilon$-critical point of $F$ as a point $\bar{x}$ whose violation is at most $\epsilon$.

\begin{definition}\label{Def: eps-critical point}
	Given $\epsilon\geq 0$, we say that the $\bar{x}\in\Omega$ is an $\epsilon$-critical point of $F$, if there exist $\alpha>0$ and $(\bar{y},\bar{z})\in \partial g(\bar{x})\times \partial h_2(\bar{x})$ such that
	\begin{equation}\label{eq: def eps-critical point}
		\dist\left(\bar{x},~ \prox_{ \alpha \bar{c}  f_{\mathcal{C}} }\left(
		\bar{x} - \alpha (\nabla h_1(\bar{x})-\bar{z} - \bar{c}^2 f(\bar{x}) \bar{y})\right)\right)
		\leq \epsilon
	\end{equation}
	holds with $\bar{c} = 1/g(\bar{x})$ and $f_{\mathcal{C}} := f+\iota_{\mathcal{C}}$.
\end{definition}

In the next section, inspired by the min–max reformulation, we design an alternating maximization proximal descent algorithm (AMPDA) and, under a mild assumption, establish its subsequential convergence to a critical point of $F$ as well as its iteration complexity for finding an $\epsilon$-critical point.


\section{An alternating maximization proximal descent algorithm} \label{section: AMDPA}

In this section, we propose an alternating maximization proximal descent algorithm (AMPDA) for solving the min-max reformulation \eqref{problem: fractional min-max}, and establish its subsequential convergence to a critical point as well as its iteration complexity for returning an $\epsilon$-critical point of $F$ under a mild assumption.

\subsection{AMPDA and its interpretations} 
In this subsection, we present AMPDA and then make some interpretations on this algorithm.
To this end, we first introduce an auxiliary  $Q:\mathbb{R}^n\times\mathbb{R}^n\times\mathbb{R}^n\times\mathbb{R}\to(-\infty, +\infty]$, which is defined by 
\begin{equation}\label{eq: def Q}
	Q(x,y,z,c) = \iota_{\mathcal{C}}(x) +  2c f(x) + c^2 f(x)  (g^*(y)-\innerP{x}{y}) + h_1(x)+h_2^*(z)-\innerP{x}{z}
\end{equation}
if $(y,z)\in\dom(g^*)\times\dom(h_2^*)$, and $Q(x,y,z,c)=+\infty$ otherwise. Let $F$ and $\wF$ be defined by \eqref{eq: def F} and \eqref{problem: fractional min-max}, respectively. Then, for all $x\in\Omega\cap\mathcal{C}$, $(y,z)\in \dom(g^*)\times \dom(h_2^*)$ and $c_x = 1/g(x)$, it holds that
\begin{equation}\label{eq: relation of F and Q}
	F(x) \overset{\StatementNumUpperCase{1}}{=} \wF(x,c_x) \overset{\StatementNumUpperCase{2}}{\leq} Q(x,y,z,c_x),
\end{equation}
where $\StatementNumUpperCase{1}$ follows from \eqref{eq: F = max_c wF}, $\StatementNumUpperCase{2}$ follows from the \textit{Fenchel-Young Inequality} and the non-negativity of $c_x^2f(x)$, and $\StatementNumUpperCase{2}$ becomes an equality when $(y,z)\in\partial g(x)\times\partial h_2(x)$. Now we formally present our AMPDA in Algorithm \ref{alg:proposed}.
Before proceeding, we provide some interpretations and observations on AMPDA.
This algorithm can be interpreted as an alternating iteration method for solving the min-max reformulation \eqref{problem: fractional min-max} as illustrated below.

\begin{algorithm}
	\caption{AMPDA for solving problem \eqref{problem:root}.}
	\label{alg:proposed}
	\renewcommand{\arraystretch}{1.2}
	\begin{tabular}{ll}
		\textbf{Step 0.} & Input $x^0\in\dom(F)$, 
		$0<\underline{\alpha} \leq  \overline{\alpha}$, $\sigma>0$, $0<\gamma<1$, and set  $k \leftarrow 0$. \\[5pt]
		\textbf{Step 1.} & Compute $c_k = 1 / g(x^k)$.\\
		& Choose $(y^k,z^k) \in \partial g(x^k)\times \partial h_2(x^k)$.\\
		& Set $\alpha := \widetilde{\alpha}_k \in [\underline{\alpha},\overline{\alpha}]$.\\ [5pt]
		\textbf{Step 2.} & 
		Compute 
	\end{tabular}
	\begin{equation}\label{eq: proximal-step of MM-PGSA}
		\widehat{x}^k \in \prox_{ \alpha c_k  f_{\mathcal{C}} }\left(
		x^k - \alpha (\nabla h_1(x^k)-z^k - c_k^2 f(x^k) y^k)\right).
	\end{equation}
	\begin{tabular}{ll}
		\textbf{~~~~~~~~~} & If $\widehat{x}^k\in\Omega$ and satisfies
	\end{tabular}%
	\begin{equation}\label{Linesearch: Rule strong}
		Q(\widehat{x}^k,y^k,z^k,1/g(\widehat{x}^k)) + \frac{\sigma}{2}\|\widehat{x}^k-x^k\|_2^2 \leq F(x^k);
	\end{equation}
	\begin{tabular}{ll}
		& Then, set $x^{k+1} = \widehat{x}^k$ and go to \textbf{Step 3};\\
		& Else, set $\alpha := \alpha \gamma$ and repeat \textbf{Step 2};\\
		\textbf{Step 3.} & Record $\alpha_k := \alpha$, set $k\leftarrow k+1$, and go to \textbf{Step 1}. 
	\end{tabular}
\end{algorithm}

First, for fixed $x \in \Omega \cap \mathcal{C}$, $\wF(x,\cdot)$ is either a concave one-dimensional quadratic function or a constant function, which always admits a global maximizer at $c=1/g(x)$. Hence, given $x^k\in\Omega\cap\mathcal{C}$, AMPDA sets $c_k = 1/g(x^k)$ to maximize $\wF(x^k,\cdot)$ in Step 1. 

Second, in Step 2 of the $k$-th iteration, AMPDA solves a proximal subproblem with respect to $f_{\mathcal{C}}$, which also requires to evaluate the gradient of $h_1$, the subgradients of $g$ and $h_{2}$ at $x^{k}$. As we will see below, this can be viewed as applying one step of the majorization-minimization scheme to the minimization problem $\min\{\wF(x,c_k):x\in\mathbb{R}^n  \}$. Since $f$ and $\nabla h_{1}$ are locally Lipschitz continuous on $\Omega$, there exists $\delta_{k}>0$ such that $f$ and $\nabla h_{1}$ are Lipschitz continuous on the neighborhood $\mathcal{B}(x^{k}, \delta_{k})\subseteq \Omega$. Let $L_{f, k}>0$  and $L_{\nabla h_1,k}>0$ be the associated Lipschitz moduli of $f$ and $\nabla h_1$ respectively. In the majorization stage, we introduce the surrogate function $\mathcal{M}(\cdot \mid x^{k}, c_{k})$ of $\wF(\cdot,c_k)$ over $\mathcal{B}(x^{k}, \delta_{k})$, which is in the form of 
\begin{multline}\label{eq: majorization surrogate of tilde_F}
	\mathcal{M}(x \mid x^{k}, c_{k})=(2 c_{k}-c_{k}^{2} g(x^{k})) f(x)+\left\langle\nabla h_{1}(x^{k})-c_{k}^{2} f(x^{k}) y^{k}-z^{k}, x-x^{k}\right\rangle \\
	+\left(c_{k}^{2} L_{f, k}\|y^{k}\|_{2}+\frac{L_{\nabla h_{1}, k}}{2}\right)\|x-x^{k}\|_{2}^{2}+h(x^{k})+\iota_{\mathcal{C}}(x).
\end{multline}
Clearly, $\mathcal{M}(x^{k} \mid x^{k}, c_{k})=\wF(x^k, c_{k})$. Moreover, for all  $x \in \mathcal{B}(x^{k}, \delta_{k})$, we have 
\begin{align*}
	& \wF(x, c_{k})\overset{\StatementNumUpperCase{1}}{=}\iota_{\mathcal{C}}(x)+f(x)\left(2 c_{k}-c_{k}^{2} g(x)\right)+h(x) \\[3pt]
	& \overset{\StatementNumUpperCase{2}}{\leq} \iota_{\mathcal{C}}(x)+f(x)\left(2 c_{k}-c_{k}^{2}\left(g(x^{k})+\left\langle y^{k}, x-x^{k}\right\rangle\right)\right)+h(x) \\[3pt]
	& =\iota_{\mathcal{C}}(x)+\left(2 c_{k}-c_{k}^{2} g(x^{k})\right) f(x)
	-\left\langle c_{k}^{2} f(x^k) y^{k}, x-x^{k}\right\rangle +h(x) \\[3pt]
	&\quad +\left\langle c_{k}^{2}  y^{k}, x-x^{k}\right\rangle (f(x^k)-f(x)) \\[3pt]
	& \overset{\StatementNumUpperCase{3}}{\leq} \iota_{\mathcal{C}}(x)+\left(2 c_{k}-c_{k}^{2} g(x^{k})\right) f(x)-\left\langle c_{k}^{2} f(x^{k}) y^{k}, x-x^{k}\right\rangle+h(x) \\[3pt]
	&\qquad +c_{k}^{2} L_{f, k}\|y^{k}\|_2\|x-x^{k}\|_{2}^{2} \overset{\StatementNumUpperCase{4}}{\leq} \mathcal{M}(x \mid x^{k}, c_{k}),
\end{align*}
where $\StatementNumUpperCase{1}$ follows from the definition of $\widetilde{F}$ (see \eqref{problem: fractional min-max}); $\StatementNumUpperCase{2}$ follows from $f \geq 0$ and $y^{k} \in \partial g(x^{k})$; $\StatementNumUpperCase{3}$ follows from the Lipschitz continuity of $f$ and the \textit{Cauchy-Schwarz Inequality}, and $\StatementNumUpperCase{4}$ follows from the inequality $h(x) \leq h(x^{k})+\left\langle\nabla h_{1}(x^{k})-z^{k}, x-x^{k}\right\rangle+\frac{L_{ \nabla h_{1},k}}{2}\|x-x^{k}\|_{2}^{2}$.  This inequality holds because the convexity of $h_2$ implies $h_2(x) \geq h_2(x^k) + \left\langle z^{k}, x-x^{k}\right\rangle$, while the \textit{Descent Lemma} applied to the locally Lipschitz differentiable function $h_{1}$ yields $h_{1}(x) \leq h_1(x^k) + \left\langle\nabla h_{1}(x^{k}), x-x^{k}\right\rangle+\frac{L_{ \nabla h_{1},k}}{2}\|x-x^{k}\|_{2}^{2}$.
Next, in the minimization stage one naturally updates $x$ by solving $\min\{\mathcal{M}(x \mid x^{k}, c_{k}):x\in\mathbb{R}^n \}$. However, as the Lipschitz moduli $L_{f, k}$  and $L_{\nabla h_1,k}$ are generally difficult to estimate, we use $1/ (2 \alpha)$ with $\alpha>0$ to replace the unknown term $c_{k}^{2} L_{f, k}\|y^{k}\|_2+\frac{L_{ \nabla h_{1},k}}{2}$ in $\mathcal{M}(x \mid x^{k},c_k)$, 
where $\alpha$ is determined by a backtracking line search procedure. Invoking this, \eqref{eq: majorization surrogate of tilde_F} and the fact that $c_k=1/g(x^k)$, we actually in the minimization stage construct $x^{k+1}$ via
\begin{multline} \label{eq: proximal-step of MM-PGSA 2}
	x^{k+1} \in \arg\min \Big\{c_k f(x)+\left \langle\nabla h_1(x^{k})- c_k^2 f(x^{k}) y^{k}-z^{k}, x-x^{k}\right\rangle\\
	+\frac{1}{2 \alpha}\|x-x^{k}\|_{2}^{2}:x\in\mathcal{C} \Big\}.
\end{multline}
It is obvious that \eqref{eq: proximal-step of MM-PGSA 2} is exactly equivalent to \eqref{eq: proximal-step of MM-PGSA} in Step 2 of AMPDA.

Third, the line search procedure is stopped once $\widehat{x}^{k} \in \Omega$ and \eqref{Linesearch: Rule strong} is satisfied. Using the fact that $(y^k,z^k)\in \partial g(x^k)\times \partial h_2(x^k)$ and the Fenchel-Young equality, it is easy to verify that
\begin{multline}
	Q(\widehat{x}^k,y^k,z^k,1/g(\widehat{x}^k)) = 
	\frac{f(\widehat{x}^k)}{g^2(\widehat{x}^k)} (2g(\widehat{x}^k) - g(x^k)) + h_1(\widehat{x}^k)-h_2(x^k)\\
	+\innerP{x^k-\widehat{x}^k}{z^k+\frac{f(\widehat{x}^k)}{g^2(\widehat{x}^k)}y^k}.
\end{multline}
Hence, in practice we do not need to evaluate the conjugate terms $g^*$ and $h^*_2$ in the line-search procedure. Also, invoking $(y^k,z^k)\in\dom(g^*)\times\dom(h_2^*) $ and \eqref{eq: relation of F and Q}, we immediately see that the condition \eqref{Linesearch: Rule strong} is generally stronger than the sufficient descent condition of $F$ at $\widehat{x}^k$, that is,
\begin{equation}\label{eq: z10091452}
	F(\widehat{x}^{k})+\frac{\sigma}{2}\|\widehat{x}^{k}-x^{k}\|_{2}^{2} \leq F(x^{k}).
\end{equation}
Although the subsequential convergence can be still guaranteed if the terminating condition \eqref{Linesearch: Rule strong} of the line search scheme in AMPDA is replaced by \eqref{eq: z10091452}, it is difficult to show the convergence of the full solution sequence generated by this altered algorithm under the KL assumption on $F$. This is because it is intractable to verify that $F$ satisfies the relative error condition (condition \StatementNum{2} in Proposition \ref{ppsition:3.4-3}) at the solution sequence without additional smooth assumption on $g$. In contrast, we can use $Q$ defined in \eqref{eq: def Q} as an auxiliary function, and prove that it generally satisfies conditions \StatementNum{1}-\StatementNum{3} of Proposition \ref{ppsition:3.4-3} at the sequence $\{(x^{k+1}, y^{k}, z^{k}, c_{k+1}) : k \in \mathbb{N}\}$ generated by AMPDA equipped with the line-search condition \eqref{Linesearch: Rule strong}, as shown in Section \ref{section: Global convergence analysis}. Due to these and invoking Proposition \ref{ppsition:3.4-3}, we immediately obtain the sequential convergence of AMPDA under the KL assumption on $Q$.

\subsection{Convergence analysis for the AMPDA}

In this subsection, we conduct convergence analysis for AMPDA. To this end, we first introduce the following basic assumption concerning its initial point.

\begin{assumption}\label{assumption: X is compact}
	The level set $\mathcal{X}:= \{ x\in\dom(F): F(x)\leq F(x^0) \}$ is compact.
\end{assumption}

Assumption \ref{assumption: X is compact} requires that the set $\mathcal{X}$ is bounded and closed. The boundedness of the level set associated with the extended objective is a very standard assumption in nonconvex optimization. In particular, for problem \eqref{problem:root}, the boundedness of $\mathcal{X}$ is automatically guaranteed if the constraint set $\mathcal{C}$ is bounded or the function $h$ is coercive (i.e., $\lim_{\|x\|_2 \to \infty} h(x) = +\infty$). However, the closedness of $\mathcal{X}$ often requires a more careful discussion since $F$ is probably not closed. In view of \cite[Proposition 4.4]{NaZhang-QiaLi:2022SIAM-OPT}, and using the closedness of $h$ and $\mathcal{C}$, the objective $F$ is a closed function  whenever $f$ and $g$ do not attain zero simultaneously. Consequently, any level set of $F$, including $\mathcal{X}$, is closed. In the case where $f$ and $g$ can attain zero simultaneously, we can still guarantee the closedness of $\mathcal{X}$ via appropriately selecting $x^0$, as shown in the following proposition.

\begin{proposition}\label{Prop: sufficent for X closed}
	Suppose that  $\mathcal{O}:=\{ x\in\mathbb{R}^n: f(x)=g(x)=0 \} \neq \emptyset$. Then, $\mathcal{X}$ is closed, if $x^0\in\dom(F)$ satisfies
	\begin{equation}\label{eq: sufficient condition for the closedness of X}
		F(x^0) < \inf\left\{ \Liminf_{z\to x} F(z): x\in\mathcal{O} \right\}.
	\end{equation}  
\end{proposition}

\begin{proof}
	We shall show that $\mathcal{X}$ is closed by verifying that each accumulation point of $\mathcal{X}$ belongs to $\mathcal{X}$ when \eqref{eq: sufficient condition for the closedness of X} is satisfied. Let $x^{\star}$ be an accumulation point of $\mathcal{X}$ and the sequence $\{x^k:k\in\mathbb{N} \} \subseteq \mathcal{X}$ converges to  $x^{\star}$. For any $k\in\mathbb{N}$, we derive from $x^k\in\mathcal{X}$ that $h(x^k)+f(x^k)/g(x^k) \leq F(x^0)$. This together with $g(x^k)>0$ yields that, for any $k\in\mathbb{N}$,
	\begin{multline}\label{eq: z12241425}
		f(x^{k})+h(x^k) g(x^{\star}) \leq F(x^0)g(x^{k}) + h(x^k)(g(x^{\star})-g(x^{k}))\\
		\leq F(x^0)g(x^{k}) + \sup\{|h(x^k)|:k\in\mathbb{N} \}|g(x^{\star})-g(x^{k})|. 
	\end{multline}
	Firstly, we have $\inf\{h(x^k):k\in\mathbb{N} \}>-\infty$ from $x^k\to x^{\star}$ and the closedness of $h$. Secondly, $\sup\{h(x^k):k\in\mathbb{N} \} = \sup\{F(x^k)-f(x^k)/g(x^k):k\in\mathbb{N} \}\leq F(x^0)$ follows from $x^k\in\mathcal{X}$ for each $k\in\mathbb{N}$, and the fact that $F\leq F(x^0)$ and $f/g\geq 0$ holds on $\mathcal{X}$. In view of these two reasons, we deduce that $\sup\{|h(x^k)|:k\in\mathbb{N} \}<+\infty$. Therefore, by passing to the limit on the both sides of \eqref{eq: z12241425} with $k\to\infty$, we derive that
	\begin{equation}\label{eq: z08121853}
		f(x^{\star})+h(x^{\star}) g(x^{\star}) \leq F(x^0)g(x^{\star}),
	\end{equation}
	owing to the closed $f$, the closed $g(x^{\star})h$, and the continuous $g$. From \eqref{eq: z08121853} and the closedness of $\mathcal{C}$, one can derive $x^{\star}\in\mathcal{X}$ if $g(x^{\star}) >0$. In fact, $g(x^{\star}) >0$ is guaranteed when $x^0\in\dom(F)$ satisfies \eqref{eq: sufficient condition for the closedness of X}. Assume on the contrary that \eqref{eq: sufficient condition for the closedness of X} holds but $g(x^{\star}) = 0$, which forces $f(x^{\star}) = 0$ according to \eqref{eq: z08121853} and $f\geq 0$. In other words, $g(x^{\star}) = 0$ leads to $x^{\star} \in \mathcal{O}$. This, along with \eqref{eq: sufficient condition for the closedness of X}, implies that $F(x^0) <\Liminf_{k\to \infty} F(x^k)$, which contradicts $\{x^k:k\in\mathbb{N} \} \subseteq \mathcal{X}$. 
	\hfill $\square$ \end{proof}

%

We remark that for both problems \eqref{problem: L1dL2} and \eqref{problem: L1dSK}, the zero vector is the unique point at which both $f$ and $g$ vanish. In view of Proposition \ref{Prop: sufficent for X closed} and these two problems, we deduce that if $x^{0} \in \mathbb{R}^{n}$ satisfies
\begin{equation}\label{eq: requirement of x0}
	F(x^{0})  < \Liminf_{z \to 0} F(z)=1+\frac{\lambda}{2}\|b-\mathcal{T}_{\mu}(b)\|_{2}^{2},
\end{equation}
then the level set $\mathcal{X}$ is closed. This together with the boundedness of the constraint set $\{x \in \mathbb{R}^{n}: \underline{x} \leq x \leq \overline{x}\}$ shows that Assumption \ref{assumption: X is compact} is fulfilled for problems \eqref{problem: L1dL2} and \eqref{problem: L1dSK} if the initial point $x^{0}$ is properly chosen. Moreover, as we shall demonstrate in Section \ref{section: Numerical experiments}, such an initial point $x^0$ satisfying \eqref{eq: requirement of x0} can be easily constructed in practice. At this stage, however, we concentrate on the convergence analysis of the proposed algorithm.

With the aid of Assumption \ref{assumption: X is compact}, we now establish the following technical lemma, which will be frequently used in our subsequent analysis.

\begin{lemma}\label{Lemma: Lipschitz continuous}
	Suppose that Assumption \ref{assumption: X is compact} holds. Then, we have $m_g := \inf \{g(x):x\in\mathcal{X}  \}>0$. Moreover, there exists $\Delta>0$ such that the following statements hold with the set  $\mathcal{X}_{\Delta}$ defined by
	\begin{equation*}\label{eq: def the set X_Delta}
		\mathcal{X}_{\Delta}:=\{ x\in\mathbb{R}^n: \text{$\|x-\bar{x}\|_2\leq \Delta$ for some $\bar{x}\in\mathcal{X}$} \}.
	\end{equation*}
	\begin{enumerate}[\upshape(\romannumeral 1 )]
		\item For any $x\in\mathcal{X}_{\Delta}$, there holds $g(x)\geq m_g/2$.
		\item $f$, $g$, $1/g$ and $\nabla h_1$ are globally Lipschitz continuous on $\mathcal{X}_{\Delta}$.
		\item The quantities defined below are finite.
		\begin{align*}
			&M_{\partial g} := \sup\{\|y\|_2:y\in\partial g(x) \text{~for some $x\in\mathcal{X}$}  \},\\
			&M_{f/ g^2} := \sup\{ f(x)/g^2(x):  x\in\mathcal{X}_{\Delta}  \}, ~~~~~ M_g := \sup\{ g(x):  x\in\mathcal{X}  \}.
		\end{align*}
	\end{enumerate}
\end{lemma}

\begin{proof}
	Since the continuous $g$ is strictly positive on the compact $\mathcal{X}$, we deduce that $m_g>0$ and the statement \StatementNum{1} holds. Then,  $f$, $g$, and $\nabla h_1$ are globally Lipschitz continuous on $\mathcal{X}_{\Delta}$, since they are locally Lipschitz continuous around each point in the compact $\mathcal{X}_{\Delta}\subseteq \Omega$ (\cite[Chapter 1, 7.5 Exercise (c)]{Clarke:2008BookNonsmooth_analysis_and_control_theory}). This, along with the statement \StatementNum{1}, implies that the statement \StatementNum{2} holds. Lastly, $M_{\partial g}$, $M_{f/ g^2}$, and $M_g$ are finite due to the boundedness of $\cup_{x\in\mathcal{X}}\partial g(x)$ (\cite[Proposition 5.4.2]{Bertsekas:Convexoptimizationtheory}), continuity of $f/g^2$ on the compact $\mathcal{X}_{\Delta}$, and the continuity of $g$ on the compact $\mathcal{X}$, respectively. This completes the proof.
	\hfill $\square$ \end{proof}

Next, we prove the well-definedness of AMPDA, which also implies its sufficient descent property.
\begin{proposition}\label{Lemma: well-defined of the Alg}
	Suppose Assumption \ref{assumption: X is compact} holds. Then AMPDA is well-defined. Specifically, there exists $\underline{\alpha}_{\sigma}>0$ such that, for all $k\in\mathbb{N}$, Step 2 of AMPDA terminates at some $\alpha_k\geq \AlphaStar$. Moreover, the sequence $\{x^k: k\in\mathbb{N} \}$ generated by AMPDA falls into the level set $\mathcal{X}\subseteq \Omega\cap\mathcal{C}$ and satisfies
	\begin{equation}\label{eq: z05201743}
		F(x^{k+1}) + \frac{\sigma}{2}\|x^{k+1}-x^k\|^2_2 \leq F(x^k),
	\end{equation} 
	where $\sigma>0$ is given in Step 0 of AMPDA.
\end{proposition}

\begin{proof}
	We prove this proposition by induction. It is clear that $x^0\in\mathcal{X}$ and we assume that $x^k\in\mathcal{X}$ for some $k\geq 0$. 
	
	Let $\mathcal{X}_{\Delta}$ be given as Lemma \ref{Lemma: Lipschitz continuous}.
	We first show that there exists $\underline{\alpha}_{\Delta}>0$ such that $\widehat{x}^k$ falls into the set $\mathcal{X}_{\Delta}\cap\mathcal{C}$, whenever $\alpha \in(0, \underline{\alpha}_{\Delta}]$.
	From \eqref{eq: proximal-step of MM-PGSA}, we know that 
	\begin{multline}\label{eq: z08181641}
		\alpha c_kf(\widehat{x}^k)+\frac{1}{2}\|\widehat{x}^k-x^k\|_{2}^{2}+\left\langle  \widehat{x}^k-x^k, \alpha(\nabla h_{1}(x^k)-z^k-c^{2}_k f(x^k) y^k)\right\rangle \\ \leq \alpha c_k f(x^k).
	\end{multline}
	The relation \eqref{eq: z08181641} implies that $ \alpha c_kf(\widehat{x}^k)+\frac{1}{2}\|\widehat{x}^k-x^k\|_{2}^{2}-\alpha \|\widehat{x}^k-x^k\|_{2}\|\nabla h_{1}(x^k)-z^k-c^{2}_k f(x^k) y^k\|_{2} \leq \alpha c_k f(x^k)$, which is a quadratic inequality of the term $\|\widehat{x}^k-x^k\|_{2}$. In view of this, $c_k=1/g(x^k)>0$ and $f(\widehat{x}^k) \geq 0$, we deduce that
	\begin{multline}\label{eq: z08181614}
		\|\widehat{x}^k-x^k\|_{2} \leq \alpha|g(x^k)|\left\|\nabla h_{1}(x^k)-z^k-c^{2}_k f(x^k) y^k\right\|_{2} \\
		+\sqrt{\alpha^{2}|g(x^k)|^{2}\|\nabla h_{1}(x^k)-z^k-c^{2}_k f(x^k) y^k\|_{2}^{2}+2 \alpha f(x^k)}.
	\end{multline}
	Since the compact $\mathcal{X} \subseteq \Omega$ implies that $\sup \{|g(x)|\|\nabla h_{1}(x)-z-c^{2} f(x) y\|_{2}: x \in \mathcal{X}, (y, z) \in \partial g(x) \times \partial h_{2}(x), c=1 / g(x)\}<+\infty$ and $\sup \{f(x): x \in \mathcal{X}\}<+\infty$, the relation \eqref{eq: z08181614} along with $x^k\in\mathcal{X}$ implies that the term $\|\widehat{x}^k-x^k\|_{2}$ can be narrow down by the parameter $\alpha$. Therefore, there exists some $\underline{\alpha}_{\Delta}>0$ such that 
	\begin{equation}\label{eq: z10071937}
		\|\widehat{x}^k-x^k\|_{2} \leq \Delta/2, \text{~~~~~for any $\alpha \in(0, \underline{\alpha}_{\Delta}]$}.
	\end{equation}
	This, together with $x^k\in\mathcal{X}$ and \eqref{eq: proximal-step of MM-PGSA}, yields that  $\widehat{x}^k\in\mathcal{X}_{\Delta} \cap \mathcal{C}$ when $\alpha \in(0, \underline{\alpha}_{\Delta}]$.

	Next, we shall show that there exists $\underline{\alpha}_{\sigma} \in (0,\underline{\alpha}_{\Delta}]$ such that $\widehat{x}^k\in\Omega$ and satisfies \eqref{Linesearch: Rule strong} stated in Step 2 of AMPDA whenever $\alpha \in (0,\underline{\alpha}_{\sigma}]$. Let $\alpha \in(0, \underline{\alpha}_{\Delta}]$. Clearly, with the help of Lemma \ref{Lemma: Lipschitz continuous} \StatementNum{1}, the statement that $\widehat{x}^k\in\Omega$ holds due to $\widehat{x}^k\in\mathcal{X}_{\Delta}$. From \eqref{eq: z10071937}, we have $\widehat{x}^k \in \mathcal{B}(x^k, \Delta)$. This, together with the Lipschitz continuity of $\nabla h_{1}$ on $\mathcal{B}(x^k, \Delta) \subseteq \mathcal{X}_{\Delta}$ (see Lemma \ref{Lemma: Lipschitz continuous} \StatementNum{2}), implies that
	\begin{equation}\label{eq: z05181655}
		h_1(\widehat{x}^k)\leq h_1(x^k) + \innerP{\nabla h_1(x^k)}{\widehat{x}^k-x^k} + \frac{L_{\nabla h_1}}{2}\|\widehat{x}^k-x^k\|_2^2,
	\end{equation}
	where $L_{\nabla h_1}>0$ denotes the Lipschitz modulus of $\nabla h_{1}$ on $\mathcal{X}_{\Delta}$. Multiplying $1/\alpha$ on the both sides of \eqref{eq: z08181641} and summing \eqref{eq: z05181655}, we derive that
	\begin{multline}\label{eq: z08181737}
		c_k f(\widehat{x}^k)+\left\langle x^k-\widehat{x}^k, z^k+c^{2}_k f(x^k) y^k\right\rangle + h_{1}(\widehat{x}^k)\\+ c_k\left(\frac{1}{2 \alpha}-\frac{L_{\nabla h_1}}{2}\right)\|\widehat{x}^k-x^k\|_{2}^{2}
		\leq c_k f(x^k) + h_{1}(x^k).
	\end{multline}
	In addition, by letting $\hat{c}=1 / g(\widehat{x}^k)$, the direct computation yields that
	\begin{align}\label{eq: z08181736}
		& c_k f(\widehat{x}^k)+\left\langle x^k-\widehat{x}^k, c^{2}_k f(x^k) y^k\right\rangle \\
		&\qquad -\Big( \hat{c}^{2} f(\widehat{x}^k)\left(2 g(\widehat{x}^k)-g(x^k)\right)+\left\langle x^k-\widehat{x}^k, \hat{c}^{2} f(\widehat{x}^k) y^k\right\rangle\Big) \notag \\
		&= c_k \hat{c}^{2} f(\widehat{x}^k)\left(g(\widehat{x}^k)-g(x^k)\right)^{2}+\left\langle x^k-\widehat{x}^k,\left(c^{2}_k f(x^k)-\hat{c}^{2} f(\widehat{x}^k)\right) y^k\right\rangle \notag\\
		&\stackrel{\StatementNumUpperCase{1}}{\geq}
		-\left(c_k \hat{c}^{2} f(\widehat{x}^k) L_{g}^{2}+L_{f / g^{2}}\|y^k\|_2\right)\|x^k-\widehat{x}^k\|_{2}^{2} \notag\\
		& \stackrel{\StatementNumUpperCase{2}}{\geq}-c_k\left(M_{f / g^{2}} L_{g}^{2}+L_{f / g^{2}} M_{\partial g} M_{g}\right)\|x^k-\widehat{x}^k\|_{2}^{2},\notag
	\end{align}
	where $\StatementNumUpperCase{1}$ follows from the \textit{Cauchy-Schwarz Inequality} and Lemma \ref{Lemma: Lipschitz continuous} \StatementNum{2} with the positive $L_{g}$ and $L_{f/ g^2}$ being the Lipschitz moduli of $g$ and $f/g^2$ on $\mathcal{X}_{\Delta}$, respectively, and $\StatementNumUpperCase{2}$ follows from Lemma \ref{Lemma: Lipschitz continuous} \StatementNum{3}. Combining \eqref{eq: z08181737} and \eqref{eq: z08181736}, we derive that
	\begin{multline}\label{eq: z08181753}
		\hat{c}^{2} f(\widehat{x}^k)\left(2 g\left(\widehat{x}^k\right)-g(x^k)\right)+\left\langle x^k-\widehat{x}^k, z^k+\hat{c}^{2} f(\widehat{x}^k) y^k\right\rangle 
		+h_{1}(\widehat{x}^k)\\+c_k\left(\frac{1}{2 \alpha}-\frac{L_{\nabla h_1}}{2}-M_{f / g^{2}} L_{g}^{2}-L_{f / g^{2}} M_{\partial g} M_{g}\right)\left\|\widehat{x}^k-x^k\right\|_{2}^{2} 
		\\
		\leq c_k f(x^k)+h_{1}(x^k).
	\end{multline}
	Recalling $\widehat{x}^k\in\mathcal{X}_{\Delta}\cap\mathcal{C}$, and	
	invoking \eqref{eq: def Q}, $\innerP{x^k}{y^k} = g(x^k)+g^*(y^k)$ and $\innerP{x^k}{z^k} = h_2(z^k)+h_2^*(z^k)$, one can verify that \eqref{eq: z08181753} is equivalent to 
	\begin{multline*}
		Q(\widehat{x}^k,y^k,z^k,1/g(\widehat{x}^k))\\ + c_k\left(\frac{1}{2 \alpha}-\frac{L_{\nabla h_1}}{2}-M_{f / g^{2}} L_{g}^{2}-L_{f / g^{2}} M_{\partial g} M_{g}\right)\|\widehat{x}^k-x^k\|_2^2 \leq F(x^k).
	\end{multline*}
	Thus, the $\widehat{x}^k$ given by \eqref{eq: proximal-step of MM-PGSA} falls into $\mathcal{X}_{\Delta}\cap\mathcal{C}\subseteq\Omega\cap\mathcal{C}$, and satisfies \eqref{Linesearch: Rule strong}, if $\alpha\in(0, \underline{\alpha}_{\sigma}]$, where $\underline{\alpha}_{\sigma}>0$ is given by
	\begin{equation*}
		\underline{\alpha}_{\sigma}:=\min \left\{\underline{\alpha}_{\Delta},\left(L_{\nabla h_1}+2 M_{f / g^{2}} L_{g}^{2}+\left(2 L_{f / g^{2}} M_{\partial g}+\sigma\right) M_{g}\right)^{-1}\right\}.
	\end{equation*} 
	This implies that Step 2 of AMPDA terminates at some $\alpha_k\geq \AlphaStar$ in the $k$-th iteration.
	
	Furthermore, according to Step 2 of AMPDA in the $k$-th iteration, 	
	the derived $x^{k+1}\in\Omega\cap\mathcal{C}$ satisfies
	\begin{equation}\label{eq: z10072052}
		Q(x^{k+1},y^k,z^k,1/g(x^{k+1})) + \frac{\sigma}{2}\|x^{k+1}-x^k\|_2^2 \leq F(x^k),
	\end{equation}
	which, together with \eqref{eq: relation of F and Q}, indicates \eqref{eq: z05201743}. In view of \eqref{eq: z05201743} and $x^k\in\mathcal{X}$, we obtain $x^{k+1}\in\mathcal{X}$. Consequently, we conclude this proposition by induction.
	\hfill $\square$ \end{proof}

Proposition \ref{Lemma: well-defined of the Alg} and the inequality \eqref{eq: z05201743} immediately lead to the next corollary. In particular, we prove that the AMPDA can find an $\epsilon$-critical point of $F$ within $\mathcal{O}(\epsilon^{-2})$ iterations.

\begin{corollary}\label{Corollary: xyz bounded and xk+1 - xk -> 0}
	Suppose that Assumption \ref{assumption: X is compact} holds. Then the following statements hold.
	\begin{enumerate}[\upshape(\romannumeral 1 )]
		\item The sequence $\{(x^k,y^k,z^k,c_k): k\in\mathbb{N} \}$ generated by AMPDA is bounded.
		\item The sequence $\{F(x^k):k\in\mathbb{N} \}$ descends monotonically with $F_{\infty} := \inf\{F(x^k):k\in\mathbb{N} \} = \lim_{k\to\infty} F(x^k)$, and we have
		\begin{equation}\label{eq: z05201744}
			\lim_{k\to\infty}\|x^{k+1}-x^k\|_2=0.
		\end{equation} 
		\item Given $\epsilon>0$, there exists a $\hat{k}\leq K$ with
		\begin{equation*}
			K:=\frac{2(F(x^0)-F_{\infty})}{\sigma\epsilon^2},
		\end{equation*}
		such that $x^{\hat{k}}$ is an $\epsilon$-critical point of $F$.
	\end{enumerate}
\end{corollary}

\begin{proof}
	First, we prove the statement \StatementNum{1}. The boundedness of $\{(x^k,c_k):k\in\mathbb{N}\}$ follows from the compactness of $\mathcal{X}$ and $c_k = 1/g(x^k)$. Also from the compactness of $\mathcal{X}$, we derive that $\{(y^k,z^k):k\in\mathbb{N}\}$ is bounded, since $g$ and $h_2$ are real-valued convex and $(y^k,z^k)\in\partial g(x^k)\times \partial h_2(x^k)$. Next we prove the statement \StatementNum{2}. Since $F$ is continuous on $\Omega\cap\mathcal{C}\supseteq\mathcal{X}$, the function $F$ is bounded below on the compact $\mathcal{X}$. This along with \eqref{eq: z05201743} yields $F_{\infty}  = \lim_{k\to\infty} F(x^k)$. Then, by passing to the limit on the both sides of \eqref{eq: z05201743} with $k\to\infty$, we derive \eqref{eq: z05201744}, and the statement \StatementNum{2} follows. Finally, we prove the statement \StatementNum{3}. For all $K>0$, we have
	\begin{equation}\label{eq: z08211841}
		F(x^K) + \frac{\sigma}{2}\sum_{k=0}^{K-1}\|x^{k+1}-x^k\|^2_2 \leq F(x^0),
	\end{equation}
	by summing up \eqref{eq: z05201743} over $k=0,1,...,K-1$. Invoking Proposition \ref{Lemma: well-defined of the Alg} and \eqref{eq: proximal-step of MM-PGSA}, we have 
	\begin{equation}\label{eq: xk+1}
		x^{k+1} \in \prox_{ \alpha_k c_{k}  f_{\mathcal{C}} }\left( x^k - \alpha_k (\nabla h_1(x^k)-z^k - c_k^2 f(x^k)y^k) \right),
	\end{equation}
	for some $\alpha_k\geq \AlphaStar$. From \eqref{eq: z08211841}, \eqref{eq: xk+1} and $F_{\infty}\leq F(x^K)$, we derive
	\begin{multline}\label{eq: 08231224}
		\min_{0\leq k \leq K-1} \dist(x^k, \prox_{ \alpha_k c_k  f_{\mathcal{C}} }(
		x^k - \alpha_k (\nabla h_1(x^k)-z^k - c_k^2 f(x^k) y^k))) \\
		\leq
		\sqrt{\frac{2}{\sigma K}(F(x^0)-F_{\infty})}.
	\end{multline}
	According to Definition \ref{Def: eps-critical point} and \eqref{eq: 08231224}, we conclude the statement \StatementNum{3}.
	\hfill $\square$ \end{proof}

Now we are ready to prove the subsequential convergence of the proposed AMPDA.

\begin{theorem}\label{Theorem: subsequential convergence}
	Suppose that Assumption \ref{assumption: X is compact} holds. Then any accumulation point of the sequence $\{x^k: k\in\mathbb{N} \}$ generated by AMPDA falls into the level set $\mathcal{X}$, and is a critical point of $F$.
\end{theorem}

\begin{proof}
	Let $x^{\star}$ be an accumulation point of $\{x^k: k\in\mathbb{N} \}$.
	In view of $\{x^k: k\in\mathbb{N} \}\subseteq \mathcal{X}$ (see Proposition \ref{Lemma: well-defined of the Alg}), we derive from the compactness of $\mathcal{X}$ that $x^{\star}\in \mathcal{X}$.
	Next, we shall show that $x^{\star}$ is a critical point of $F$. 
	Let $\mathcal{K}\subseteq \mathbb{N}$ be  a subsequence, such that $\{x^k: k\in\mathcal{K} \}$ converges to the $x^{\star}$. In view of Proposition \ref{Lemma: well-defined of the Alg}, \eqref{eq: proximal-step of MM-PGSA} and \eqref{eq: xk+1}, it holds for  all $(x,k)\in\mathcal{C}\times \mathcal{K}$ that
	\begin{multline*}
		\alpha_k c_kf(x^{k+1}) + \frac{1}{2} \|x^{k+1}-x^k\|_2^2 + \innerP{x^{k+1}-x^k}{\alpha_k(\nabla h_1(x^k) - z^k -c_k^2f(x^k) y^k)}\\
		\leq \alpha_k c_kf(x) + \frac{1}{2} \|x-x^k\|_2^2 + \innerP{x-x^k}{\alpha_k(\nabla h_1(x^k) - z^k -c_k^2 f(x^k)y^k)},
	\end{multline*}
	This, together with the fact that $\AlphaStar \leq \alpha_k \leq \overline{\alpha}$, leads to 
	\begin{multline}
		c_kf(x^{k+1}) + \frac{1}{2\overline{\alpha}}\|x^{k+1}-x^k\|_2^2  + \innerP{x^{k+1}-x^k}{\nabla h_1(x^k) - z^k -c_k^2f(x^k) y^k}\\
		\leq c_kf(x) + \frac{1}{2\AlphaStar} \|x-x^k\|_2^2 + \innerP{x-x^k}{\nabla h_1(x^k) - z^k -c_k^2f(x^k) y^k}.\label{eq: z05202041}
	\end{multline}
	Since $\{(y^k,z^k):k\in\mathcal{K}\}$ is bounded (see Corollary \ref{Corollary: xyz bounded and xk+1 - xk -> 0}) and $(y^k,z^k)\in\partial g(x^k)\times \partial h_2(x^k)$ for all $k\in\mathbb{N}$, we deduce that some subsequence $\mathcal{K}_2\subseteq\mathcal{K}$ such that 
	$\{(y^k,z^k):k\in\mathcal{K}_2\}$ converges to some $(y^{\star},z^{\star})\in\partial g(x^{\star}) \times \partial h_2(x^{\star})$. In view of \eqref{eq: z05201744} and the continuity of $f$ and $1/g$ on $\mathcal{X}$, we pass to the limit with $k\in\mathcal{K}_2$ and $k\to\infty$ on the both sides of \eqref{eq: z05202041}, and then obtain that
	\begin{multline}\label{eq: z05202141}
		c_{\star} f(x^{\star})
		\leq c_{\star} f(x) + \frac{\|x-x^{\star}\|_2^2}{2\AlphaStar}  \\ + \innerP{x-x^{\star}}{\nabla h_1(x^{\star}) - z^{\star} -c_{\star}^2 f(x^{\star}) y^{\star}},
	\end{multline}
	for all $x\in\mathcal{C}$, where $c_{\star} = 1/g(x^{\star})$. The  relation \eqref{eq: z05202141} implies that
	\begin{multline}\label{eq: z10081525}
		x^{\star}\in\arg\min\{ c_{\star}f_{\mathcal{C}}(x) +\\ \frac{1}{2\alpha} \|x-x^{\star}+\alpha (\nabla h_1(x^{\star}) - z^{\star} -c_{\star}^2 f(x^{\star}) y^{\star}) \|^2_2
		:x\in\mathbb{R}^n \}
	\end{multline}
	with $\alpha = \AlphaStar$. By invoking the generalized \textit{Fermat's Rule},  the relation \eqref{eq: z10081525} yields \eqref{eq: def lifted stationary point}, which together with the fact $x^{\star}\in\mathcal{X}\subseteq\dom(F)$ indicates that $x^{\star}$ is a critical point of $F$ according to Definition \ref{Def: lifted stationary point}. 
	\hfill $\square$ \end{proof}

To close this subsection, we show that the proximal step \eqref{eq: proximal-step of MM-PGSA} in AMPDA relates to the update formula \eqref{eq: z12131420} that is used in existing proximal algorithms for nonsmooth fractional programs, if $g$ is smooth and $h_2\equiv 0$ (i.e., $h=h_1$) in problem \eqref{problem:root}. In this case, we know that $y^k=\nabla g(x^k)$, $z^k=0$ and $g(x^k)\nabla h(x^k)= \nabla(hg)(x^k)-h(x^k)\nabla g(x^k)$. Using these facts and $c_k=1/g(x^k)$ together with the equivalence between \eqref{eq: proximal-step of MM-PGSA} and \eqref{eq: proximal-step of MM-PGSA 2}, we can express \eqref{eq: proximal-step of MM-PGSA} as
\begin{multline}\label{eq: z12200029}
	x^{k+1} \in  \arg\min\left\{f(x)+ 	\Big\langle\nabla (hg)(x^{k})- 	\left( \frac{f(x^k)}{g(x^k)}+h(x^k) \right) \nabla g(x^k),
	\right.  \\ 
	\left. x-x^{k}\Big\rangle+\frac{g(x^k)}{2 \alpha} \|x-x^{k}\|_{2}^2: x \in \mathcal{C}\right\}. 
\end{multline}
By further choosing $\alpha=g(x^k)\eta_k$ with $\eta_k>0$ in \eqref{eq: z12200029}, we immediately obtain the update formula \eqref{eq: z12131420}.


\section{Sequential convergence of AMPDA} \label{section: Global convergence analysis}

In this section, we investigate the convergence of the entire sequence generated by AMPDA for solving problem \eqref{problem:root} under the KL assumption. To this end, we first present the following proposition, which regards the sufficient descent of $Q$ defined by \eqref{eq: def Q}, and is a direct result of \eqref{eq: relation of F and Q} and the iterative procedure of AMPDA.

\begin{proposition}\label{Proposition: sufficient descent of AMPDA}
	Suppose that Assumption \ref{assumption: X is compact} holds, and let $\{(x^{k}, y^{k}, z^{k}, c_{k}) : k \in \mathbb{N}\}$ be generated by AMPDA. Then, it holds for all $k\in\mathbb{N}$ that
	\begin{equation}\label{eq: z08072031}
		Q(x^{k+1},y^k,z^k,c_{k+1})+\frac{\sigma}{2}\|x^{k+1}-x^k\|_2^2 \leq Q(x^k,y^{k-1},z^{k-1},c_{k}).
	\end{equation}  
\end{proposition}

Next we prove that the relative error condition stated in Proposition \ref{ppsition:3.4-3} \StatementNum{2} holds for $Q$ at $\{(x^{k+1}, y^{k}, z^{k}, c_{k+1}) : k \in \mathbb{N}\}$ generated by AMPDA. 

\begin{proposition}\label{Proposition: relative error conditions of AMPDA}
	Suppose that Assumption 2 holds and let $\{(x^{k}, y^{k}, z^{k}, c_{k}) : k \in \mathbb{N}\}$ be generated by AMPDA. Then there exists $b>0$ such that, for all $k \in \mathbb{N}$, 
	\begin{equation}\label{eq: z08072032}
		\dist(0,\partial Q(x^{k+1},y^k,z^k,c_{k+1}))\leq b \|x^{k+1}-x^k\|_2.
	\end{equation}
\end{proposition}

\begin{proof}
	For any $(\bar{x}, \bar{y}, \bar{z}) \in(\Omega \cap \mathcal{C}) \times \dom(\partial g^{*}) \times \dom(h_{2}^{*})$ and $\bar{c}=1 / g(\bar{x})$, we have
	\begin{multline}\label{eq: z08072230}
		\widehat{\partial} Q(\bar{x}, \bar{y}, \bar{z}, \bar{c}) \\ \supseteq 
		\left[\begin{array}{l}			
			\left(2 \bar{c}+\bar{c}^{2}(g^{*}(\bar{y})-\langle\bar{x}, \bar{y}\rangle)\right) \widehat{\partial} f_{\mathcal{C}}(\bar{x})-\bar{c}^{2} f_{\mathcal{C}}(\bar{x}) \bar{y}+\nabla h_{1}(\bar{x})-\bar{z} \\
			\bar{c}^{2} f_{\mathcal{C}}(\bar{x}) \partial g^{*}(\bar{y})-\bar{c}^{2} f_{\mathcal{C}}(\bar{x}) \bar{x} \\
			\widehat{\partial} h_{2}^{*}(\bar{z})-\bar{x} \\
			2 f_{\mathcal{C}}(\bar{x})+2 \bar{c} f_{\mathcal{C}}(\bar{x})\left(g^{*}(\bar{y})-\langle\bar{x}, \bar{y}\rangle\right)
		\end{array}\right].
	\end{multline}
	The proof of \eqref{eq: z08072230} is given in Appendix \ref{Appendix: subgradient of Q}. It is known from the proximal step \eqref{eq: proximal-step of MM-PGSA} of AMPDA and the generalized \textit{Fermat's Rule} that
	\begin{equation}\label{eq: def w_fC}
		w_{f_{\mathcal{C}}}^{k+1} := \frac{g(x^k)}{\alpha_k}(x^k-x^{k+1}-\alpha_k \nabla h_1(x^k) +\alpha_k z^k) + c_kf(x^k)y^k
		\in\widehat{\partial} f_{\mathcal{C}}(x^{k+1}).
	\end{equation}
	Then, with the help of \eqref{eq: z08072230} and the fact that $(y^k,z^k)\in\partial g(x^k)\times\partial h_2(x^k)$, we have
	\begin{equation}\label{eq: def wk+1}
		w^{k+1} := (w_{x}^{k+1}, w_{y}^{k+1}, x^{k}-x^{k+1}, w_{c}^{k+1})\in\widehat{\partial} Q\left(x^{k+1}, y^{k}, z^{k}, c_{k+1}\right),
	\end{equation}
	where $w_{x}^{k+1}, w_{y}^{k+1}$ and $w_{c}^{k+1}$ are given as below:
	\begin{align}
		w_{x}^{k+1}&:=\left(2 c_{k+1}+c_{k+1}^{2}(g^{*}(y^{k})-\langle x^{k+1}, y^{k}\rangle)\right) w_{f_{\mathcal{C}}}^{k+1}\\
		& \qquad -c_{k+1}^{2} f_{\mathcal{C}}(x^{k+1}) y^{k}+\nabla h_{1}(x^{k+1})-z^{k}, \notag \\
		w_{y}^{k+1}&:=c_{k+1}^{2} f_{\mathcal{C}}(x^{k+1})(x^{k}-x^{k+1}),\label{eq: def wy} \\
		w_{c}^{k+1}&:=2  c_{k+1} f_{\mathcal{C}}(x^{k+1})\left(g(x^{k+1})+g^{*}(y^{k})-\langle x^{k+1}, y^{k}\rangle\right).\label{eq: def wc}
	\end{align}
	
	We then show that $\|w^{k+1}\|_2\leq b  \|x^{k+1}-x^k\|_2$ for some $b>0$. It amounts to prove that $w_{x}^{k+1}$, $w_{y}^{k+1}$ and $w_{c}^{k+1}$ can be bounded by the term $\|x^{k+1}-x^{k}\|_{2}$. By invoking Lemma \ref{Lemma: Lipschitz continuous} \StatementNum{3}, we derive from \eqref{eq: def wy} that
	\begin{equation}\label{eq: z08292200}
		\|w_{y}^{k+1}\|_{2} \leq M_{f/g^2}\|x^{k+1}-x^{k}\|_{2}.
	\end{equation}
	for all $k\in\mathbb{N}$. By substituting $g^{*}(y^{k})=\langle x^{k}, y^{k}\rangle-g(x^{k})$ into \eqref{eq: def wc}, we obtain
	$w_{c}^{k+1}=2 c_{k+1} f_{\mathcal{C}}(x^{k+1}) (g(x^{k+1})-g(x^{k})+\langle x^{k}-x^{k+1}, y^{k}\rangle).$ Invoking this and Lemma \ref{Lemma: Lipschitz continuous}, we deduce that
	\begin{equation}\label{eq: z08292201}
		\left\|w_{c}^{k+1}\right\|_{2} \leq 2 M_{f/g^2} M_g (L_{g}+M_{\partial g})\|x^{k+1}-x^{k}\|_{2}, 
	\end{equation}
	for all $k \in \mathbb{N}$, where $L_g>0$ denotes the Lipschitz modulus of $g$ on $\mathcal{X}$. We next focus on the term $w_{x}^{k+1}$. We observe that $w_{x}^{k+1} $ can be rewritten as
	\begin{multline}\label{eq: z08292129}
		w_{x}^{k+1}=  \left(2 c_{k+1}+c_{k+1}^{2}(g^{*}(y^{k})-\langle x^{k+1}, y^{k}\rangle)-c_{k}\right) w_{f_{\mathcal{C}}}^{k+1} 
		\\
		+ \left(c_{k}^{2} f_{\mathcal{C}}(x^{k})-c_{k+1}^{2} f_{\mathcal{C}}(x^{k+1})\right) y^{k}+\nabla h_{1}(x^{k+1})-\nabla h_{1}(x^{k})+\frac{1}{\alpha_k}(x^{k}-x^{k+1}).
	\end{multline}
	For the term $2 c_{k+1}+c_{k+1}^{2}\left(g^{*}(y^{k})-\langle x^{k+1}, y^{k}\rangle\right)-c_{k}$ in \eqref{eq: z08292129}, we have
	\begin{align}\label{eq: z08292136}
		& \left|2 c_{k+1}+c_{k+1}^{2}\left(g^{*}(y^{k})-\left\langle x^{k+1}, y^{k}\right\rangle\right)-c_{k}\right| \\
		&=  \left|c_{k+1}^{2} c_{k}\left(2 g(x^{k+1}) g(x^{k})-g^{2}(x^{k})-g^{2}(x^{k+1})\right)+\left\langle x^{k}-x^{k+1}, c_{k+1}^{2} y^{k}\right\rangle\right| \notag \\
		&\leq  \left|c_{k+1}^{2} c_{k}\left(g(x^{k+1})-g(x^{k})\right)^{2}\right|+\left\|c_{k+1}^{2} y^{k}\right\|_2\|x^{k+1}-x^{k}\|_{2} \notag \\
		&\leq   m_{g}^{-2}\left(2 L_{g}+M_{\partial g}\right)\|x^{k+1}-x^{k}\|_{2}, \notag
	\end{align}
	where the first equality follows from $g^*(y^k) = \langle x^{k},y^{k} \rangle - g(x^k)$. In view of boundedness of $\left\{\left(x^{k}, y^{k}, z^{k}, c_{k}\right): k \in \mathbb{N}\right\}$ (see Corollary \ref{Corollary: xyz bounded and xk+1 - xk -> 0}) 
	and $\AlphaStar\leq \alpha_{k} \leq \overline{\alpha}$ (see Proposition \ref{Lemma: well-defined of the Alg}), we deduce from \eqref{eq: def w_fC} that
	$$M_{w_{f_{\mathcal{C}}}}:=\sup \{\|w_{f_{\mathcal{C}}}^{k+1}\|_{2}: k \in \mathbb{N}\}<+\infty.$$
	Invoking Lemma \ref{Lemma: Lipschitz continuous}, we derive from \eqref{eq: z08292129} and \eqref{eq: z08292136} that
	\begin{multline}
		\|w_{x}^{k+1}\|_{2} \leq M_{w_{f_{\mathcal{C}}}} m_{g}^{-2}(2 L_{g}+M_{\partial g})\|x^{k+1}-x^{k}\|_{2} \label{eq: z08292159}\\
		\quad +\left(L_{f / g^{2}} M_{\partial g}+L_{\nabla h_{1}}+ (\AlphaStar)^{-1}\right)\|x^{k+1}-x^{k}\|_{2}, 
	\end{multline}
	for each $k\in\mathbb{N}$,	where $L_{f / g^{2}}>0$ and $L_{\nabla h_{1}}>0$ are the Lipschitz moduli of $f / g^{2}$ and $\nabla h_{1}$ on $\mathcal{X}$, respectively. 
	In view of \eqref{eq: def wk+1}, \eqref{eq: z08292200}, \eqref{eq: z08292201} and \eqref{eq: z08292159}, we conclude that \eqref{eq: z08072032} holds for each $k \in \mathbb{N}$. This completes the proof.
	\hfill $\square$ \end{proof}

Now we are ready to establish the sequential convergence of AMPDA.

\begin{theorem}\label{theorem: AMPDA global convergence when g* is continuous}
	Suppose that Assumption \ref{assumption: X is compact} holds.  If the $Q$ given by \eqref{eq: def Q} is a proper closed KL function, then the sequence $\{x^{k}: k \in \mathbb{N}\}$ generated by AMPDA converges to a critical point of \eqref{problem:root}.
\end{theorem}

\begin{proof}
	Let the solution sequence $\{(x^{k}, y^{k}, z^{k}, c_{k}) : k \in \mathbb{N}\}$ be generated by AMPDA. 
	We prove the convergence of the sequence $\{x^{k} : k \in \mathbb{N}\}$ by verifying all the conditions in Proposition \ref{ppsition:3.4-3}. Corollary \ref{Corollary: xyz bounded and xk+1 - xk -> 0} \StatementNum{1} shows the boundedness of the solution sequence $\{(x^{k}, y^{k}, z^{k}, c_{k}) : k \in \mathbb{N}\}$, and Propositions \ref{Proposition: sufficient descent of AMPDA} and \ref{Proposition: relative error conditions of AMPDA} respectively show that the \textit{Sufficient Descent Condition} and the \textit{Relative Error Condition} hold for the solution sequence. Thus, it remains only to verify the \textit{Continuity Condition}. 
	
	To verify the \textit{Continuity Condition}, it suffices to show that the limit $Q_{\infty} := \lim_{k\to\infty}Q(x^{k+1},y^k,z^k,c_{k+1})$ exists and $Q\equiv Q_{\infty}$ holds on the set of all accumulation point of the solution sequence. In view of Step 2 of AMPDA and \eqref{eq: relation of F and Q}, we have
	$	F(x^{k+1})\leq Q(x^{k+1}, y^{k}, z^{k}, c_{k+1}) \leq F(x^k).$
	This together with Corollary \ref{Corollary: xyz bounded and xk+1 - xk -> 0} \StatementNum{2} leads to the existence of $Q_{\infty}$ and $Q_{\infty} = F_{\infty}$. 
	We then let $(x^{\star}, y^{\star}, z^{\star}, c_{\star})$ be an accumulation point of $\{(x^{k+1}, y^{k}, z^{k}, c_{k+1}): k \in \mathbb{N}\}$, and $\{(x^{k_{j}+1}, y^{k_{j}}, z^{k_{j}}, c_{k_{j}+1}): j \in \mathbb{N}\}$ be the subsequence converging to $(x^{\star}, y^{\star}, z^{\star}, c_{\star})$. 
	Firstly, we have $c_{\star} = \lim _{j \to \infty} c_{k_{j}+1}=\lim _{j \to \infty} 1 / g(x^{k_{j}+1})=1 / g(x^{\star})$ due to the continuity of $1 / g$ on $\mathcal{X}$. 
	Secondly, we know that $ x^{k_{j}} \to x^{\star}$ thanks to $ x^{k_{j}+1} \to x^{\star}$ and \eqref{eq: z05201744}. Thirdly, $(y^{\star}, z^{\star}) \in \partial g(x^{\star}) \times \partial h_{2}(x^{\star})$ follows from $ (y^{k_{j}},z^{k_{j}}) \to (y^{\star},z^{\star})$ with $ (y^{k_{j}},z^{k_{j}})\in \partial g(x^{k_{j}}) \times \partial h_{2}(x^{k_{j}})$ for each $j \in \mathbb{N}$. 
	Consequently, the \textit{Continuity Condition} follows from
	\begin{equation*}\label{eq: z08262129}
		Q(x^{\star}, y^{\star}, z^{\star}, c_{\star})  
		\overset{\StatementNumUpperCase{1}}{=}F(x^{\star})\overset{\StatementNumUpperCase{2}}{=}\lim _{j \to \infty} F(x^{k_{j}+1})\overset{\StatementNumUpperCase{3}}{=}F_{\infty} = Q_{\infty},
	\end{equation*}
	where $\StatementNumUpperCase{1}$ follows from \eqref{eq: relation of F and Q}, $\StatementNumUpperCase{2}$ holds since $F$ is continuous on the compact $\mathcal{X}\subseteq\Omega$, and $\StatementNumUpperCase{3}$ follows from Corollary \ref{Corollary: xyz bounded and xk+1 - xk -> 0} \StatementNum{2}.

	Combining the above arguments, we conclude that the sequence $\{x^{k} : k \in \mathbb{N}\}$ is convergent. Together with Theorem \ref{Theorem: subsequential convergence}, we further conclude that the sequence $\{x^{k} : k \in \mathbb{N}\}$ converges to a critical point of $F$. \hfill $\square$ \end{proof}

We further note that, if in Definition \ref{Def:KL_property} the function  $\phi$ can be chosen as $\phi(z) = a_0 z^{1-\theta}$ for some $a_0>0$ and $\theta \in [0,1)$, then $\varphi$ is said to satisfy the KL property with the exponent $\theta$. It is known from \cite[Corollary 16]{Bolte-Daniilidis-Lewis-Shiota:2007SIAM-OPT} that every proper closed semialgebraic function satisfies the KL property with some exponent $\theta \in [0,1)$.
Similar to the analysis in \cite{Attouch-Bolte:MP2009,Wen-Chen-Pong:2018COA,Xu-Yin:SIAMIS:13}, the convergence rate of the sequence $\{x^{k} : k \in \mathbb{N}\}$ generated by AMPDA can be estimated by the potential KL exponent of $Q$, and the following statements hold:
\begin{enumerate}[\upshape(\romannumeral 1 )]
	\item If $\theta = 0$, then the sequence converges in finitely many steps;
	\item If $\theta \in (0,1/2]$, then the sequence converges R-linearly;
	\item If $\theta \in (1/2,1)$, then the convergence is sublinear.
\end{enumerate}

Moreover, by \cite[Theorem 13.2]{Rockafellar:70}, we know that, if $g$ is positively homogeneous (i.e., $g(\alpha x) = \alpha g(x)$ holds for all $\alpha>0$ and $x\in\mathbb{R}^n$), the conjugate $g^*$ is an indicator function of a closed convex set. Using this, for problem \eqref{problem:root} we have that $c^2 f(x)g^*(y) = g^*(y)$, and 
the corresponding auxiliary function given by \eqref{eq: def Q} takes the form
\begin{multline}\label{eq: z08292229}
	Q(x, y, z, c) =  \iota_{\mathcal{C}}(x)+2 c f(x) +g^*(y)\\ -c^{2}f(x) \langle x, y\rangle 
	+h_1(x)+h_{2}^{*}(z)-\langle x, z\rangle.
\end{multline}
According to the discussions on semialgebraic functions at the end of Section \ref{section: KL property and semialgebra}, we see from \eqref{eq: z08292229} that, if the components $f$, $g$, $h_1$, $h_2$ and the constraint set $\mathcal{C}$ are semialgebraic, then $Q$ is a proper closed semialgebraic function, and hence a KL function. 
Combining this with the discussions right after Proposition \ref{Prop: sufficent for X closed}, we finally invoke Theorem \ref{theorem: AMPDA global convergence when g* is continuous} to derive the following convergence result for AMPDA when applied to the signal recovery models \eqref{problem: L1dL2} and \eqref{problem: L1dSK}.

\begin{theorem}\label{Prop: AMPDA is globally convergent for solving L1dL2}
	Consider problems \eqref{problem: L1dL2} and \eqref{problem: L1dSK}. Then the sequence $\{x^{k}: k \in \mathbb{N}\}$ generated by AMPDA for each of them converges to a critical point of the respective problem if the initial point $x^{0}$ satisfies \eqref{eq: requirement of x0}.
\end{theorem}

\section{Numerical experiments} \label{section: Numerical experiments}

In this section, we conduct some preliminary numerical experiments to evaluate the performance of our proposed AMPDA. All the experiments are conducted in MATLAB R2025a on a desktop equipped with an Intel Core i5-9500 CPU (3.00 GHz) and 16GB of RAM.

Throughout the experiments, the parameters of the proposed AMPDA are set as follows: we set $\underline{\alpha}=10^{-4}$, $\overline{\alpha}=10^{4}$, $\sigma=10^{-5}$, and $\gamma=1 /2$. we update the $\widetilde{\alpha}_{k}$ in Step 1 of AMPDA as
\begin{equation} \label{eq: def ak0}
	\widetilde{\alpha}_{k}:=\max \left\{\underline{\alpha}, \min \left\{\overline{\alpha}, \frac{\|\Delta x^{k}\|_{2}^{2}}{|\langle\Delta x^{k}, \Delta h_{1}^{k}\rangle|}\right\}\right\},
\end{equation}
where $\Delta x^{k}=x^{k}-x^{k-1}$ and $\Delta h_{1}^{k}=\nabla h_{1}(x^{k})-\nabla h_{1}(x^{k-1})$ if $k \geq 1$ and $|\langle\Delta x^{k}, \Delta h_{1}^{k}\rangle| \neq 0$; otherwise, we set $\widetilde{\alpha}_{k}=1$. The attempted stepsize in \eqref{eq: def ak0} is designed following the original idea of the Barzilai–Borwein (BB) spectral method \cite{Barzilai-Borwein:IJNA1998}. Stepsize strategies of this BB type are well known to deliver remarkable practical efficiency in numerous optimization algorithms; see, for example,
\cite{Lu-Xiao:SIAM2017,Raydan:1997SIOPT,Wright-Nowak-Figueiredo:IEEESP-09}.

For problems \eqref{problem: L1dL2} and \eqref{problem: L1dSK}, the initial point of all the tested algorithms is chosen as follows. Let $T \in \mathbb{R}^{m\times m}$ be the diagonal matrix whose $j$-th ($j=1,2,...,m$) diagonal entry is set to $T_{j,j} = 1$ if $(\mathcal{T}_{\mu}(b))_{j}\neq 0$, and $T_{j,j} = 0$ otherwise. We select $i_{\star} \in \arg\max\{|(b-T b)^{\top} a_{i}|: i \in\{1,2, \cdots, n\}\}$, where $a_i$ denotes the $i$-th column of $A$, compute
\begin{equation}\label{eq: z11151646}
	\hat{\theta}:=\frac{(b-Tb)^{\top}a_{i_{\star}}}
	{\|a_{i_{\star}}\|_{2}^{2}-\|T  a_{i_{\star}}\|_{2}^{2}},
\end{equation}
and then define the initial point $x^{0}$ by 
\begin{equation}\label{eq: def x0 in exp}
	x_{i}^{0}= \begin{cases}\arg\min\left\{|\theta-\hat{\theta}|: \theta \in[\underline{x}_{i_{\star}}, \overline{x}_{i_{\star}}]\right\}, & \text { if } i=i_{\star}, \\ 0, & \text { else. }\end{cases}
\end{equation}
In Appendix \ref{Appendix:x0}, we show that this choice of $x^{0}$ satisfies the condition \eqref{eq: requirement of x0} for both problems \eqref{problem: L1dL2} and \eqref{problem: L1dSK}, provided that  $A\in\mathbb{R}^{m\times n}~(m<n)$ has full row-rank, $b\notin \mathcal{S}_{\mu}$, and $\underline{x}<0<\overline{x}$. All the tested algorithms are terminated once
\begin{equation}\label{eq: alg_terminate}
	\frac{\|x^{k}-x^{k-1}\|_{2}}{ \| x^{k}\|_{2}}<10^{-6} .
\end{equation}

\subsection{Numerical results on the $L_1/L_2$-regularized least squares problem}
In this subsection, we applied AMPDA for solving the $L_1/L_2$-regularized least squares problem
\begin{equation}\label{problem: L1dL2-LS}
	\min \left\{\frac{\|x\|_{1}}{\|x\|_{2}}+\frac{\lambda}{2} \|Ax-b\|_2^2 : x \neq 0,~ \underline{x} \leq x \leq \overline{x},~ x \in \mathbb{R}^{n}\right\},
\end{equation}
which is a special case of problem \eqref{problem: L1dL2} with $\mu=0$. The test instances $(A,b)$ are obtained from the datasets leukemia, colon-cancer, and mpg in the \textit{LIBSVM} repository\footnote{\url{https://www.csie.ntu.edu.tw/\~cjlin/libsvmtools/datasets/}}. In particular, as suggested in \cite{Huang-Jia-Yu:2010NIPS} and following the approach in \cite[Section 4.1]{Li-Sun-Toh:2018SIOPT}, we expand the original features of the dataset mpg by an order 7 polynomial. For problem \eqref{problem: L1dL2-LS}, we fix $\lambda = 1$ and set the box constraint as $\underline{x} = -10 \cdot  \One_n$ and $\overline{x} = 10 \cdot \One_n$, where $\One_n$ denotes the $n$-dimensional vector of all ones.

We compare AMPDA with the following two methods.
The first method is the alternating direction method of multipliers (ADMM) \cite[Algorithm~3.2]{Wang-Tao-Nagy-Lou:2021SIAM-ImageScience}, which is tailored for the problem \eqref{problem: L1dL2-LS}.
Following the notation in \cite{Wang-Tao-Nagy-Lou:2021SIAM-ImageScience}, we set $\rho_1=\rho_2=\beta=\rho$, $\overline{\epsilon} = 10^{-6}$, and $\texttt{j\_max} = 5$.
For better performance, we set $\rho=2$ for leukemia, $\rho=50$ for mpg, and $\rho=10$ for colon-cancer.
The second one is the gradient descent flow method (GDFA) \cite[Algorithm~1]{Wang-Aujol-Gilboa-Lou:2024IPI}.
Using the same notation as in \cite[Algorithm~1]{Wang-Aujol-Gilboa-Lou:2024IPI}, we set $\beta = 10^{-6}$ and $\texttt{jMax}=5$.
To obtain satisfactory numerical performance, we set $\rho=5$ for leukemia, $\rho=50$ for mpg, and $\rho=10$ for colon-cancer.
We remark that the parameters of the proposed AMPDA follow the settings described at the beginning of Section \ref{section: Numerical experiments} and remain unchanged for all datasets.


Table \ref{tab:L1L2-LS-results} summarizes the computational results together with the dimensions of the matrices $A$ constructed from the three datasets.
For each method, we report the objective value (Obj), the number of iterations (Iter), and the computational time in seconds (Time), where the computational time is decomposed into two parts: the startup phase (reported in parentheses) and the iterative phase.
Unlike the proposed AMPDA, which starts almost instantaneously, both GDFA and ADMM require longer startup times.
This is mainly because these two methods perform Cholesky factorizations of certain matrices before the iterations begin, in order to accelerate solving the linear systems during the iterative phase.
From Table~\ref{tab:L1L2-LS-results}, we observe that all three algorithms eventually converge to solutions of similar objective value on all tested datasets, while AMPDA consistently requires the least computational time.
Figure \ref{fig:L1L2-LS-curves} plots the objective value versus the computational time during the early stage of the iterative phase.
From these curves, we observe that AMPDA consistently exhibits the fastest convergence among all methods.
These results clearly demonstrate the superior effectiveness of AMPDA for solving the $L_1/L_2$-regularized least squares problems considered in this subsection.

\begin{table}[htbp]
	\centering
	\caption{Results for the $L_1/L_2$-regularized least squares problem \eqref{problem: L1dL2-LS}.}
	\label{tab:L1L2-LS-results}
	\begin{tabularx}{\linewidth}{
			>{\raggedright\arraybackslash}p{0.22\linewidth} 
			>{\centering\arraybackslash}p{0.2\linewidth}   
			>{\centering\arraybackslash}X	
			>{\centering\arraybackslash}X 
			>{\centering\arraybackslash}p{0.2\linewidth}  
		}
		\toprule
		Dataset & Alg & Obj & Iter & Time \\
		\midrule
		leukemia        & AMPDA & 2.375 & 2856 & 4.35 (0.00) \\
		$m=38, n=7129$ & ADMM  & 2.379 &  994 & 268.84 (4.70) \\
		& GDFA  & 2.375 & 1575 & 286.75 (1.66) \\
		& & & &\\
		colon-cancer    & AMPDA & 5.347 & 6817 & 1.99 (0.00) \\
		$m=62, n=2000$ & ADMM  & 5.345 & 2843 & 82.03 (0.27) \\
		& GDFA  & 5.345 & 3380 & 54.80 (0.12) \\
		& & & &\\
		mpg            & AMPDA & 1.101 &  299 & 0.36 (0.00) \\
		$m=392, n=3432$& ADMM  & 1.103 &  116 & 8.87 (0.61) \\
		& GDFA  & 1.101 &  181 & 7.55 (0.16) \\
		\bottomrule
	\end{tabularx}
\end{table}

\begin{figure}[htbp]
	\centering
	\begin{subfigure}{0.9\textwidth}
		\centering
		\includegraphics[width=\textwidth]{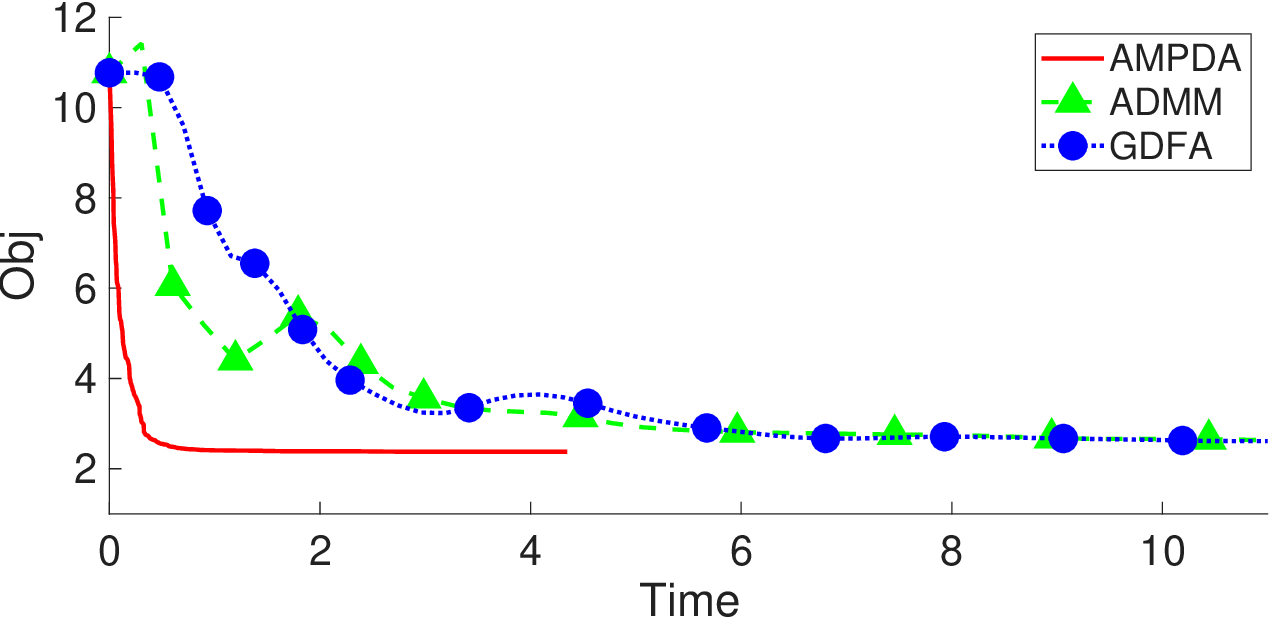}
		\caption{leukemia}
		\label{fig:L1L2-LS-leukemia}
	\end{subfigure}
	\hfill
	\begin{subfigure}{0.9\textwidth}
		\centering
		\includegraphics[width=\textwidth]{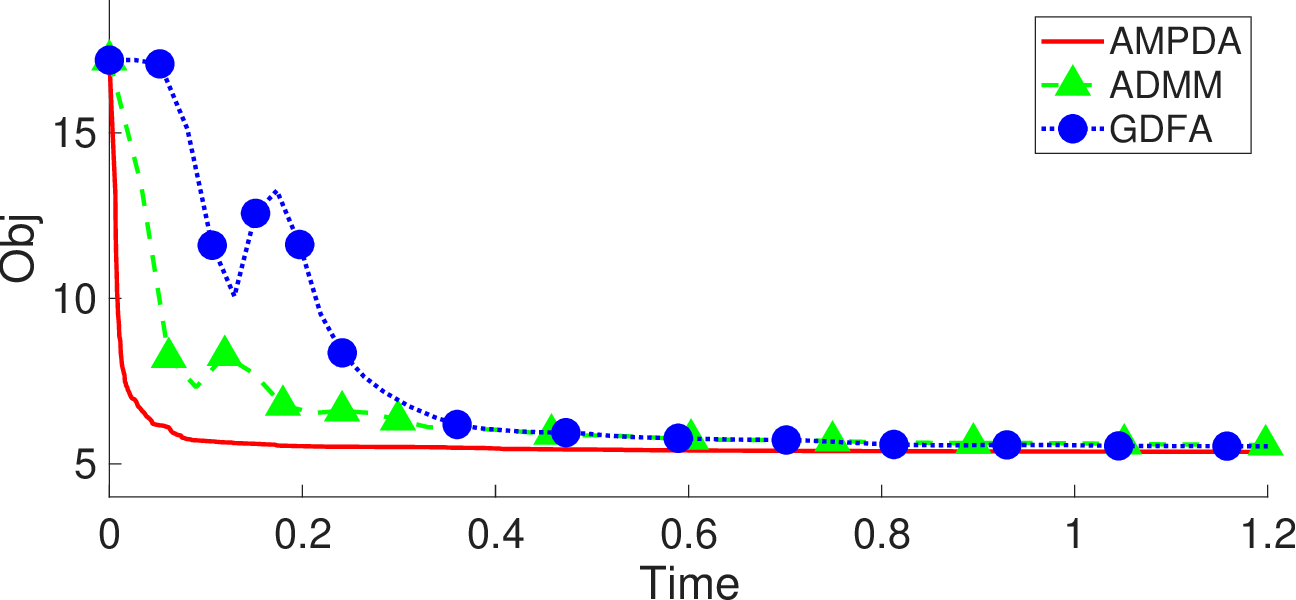}
		\caption{colon-cancer}
		\label{fig:L1L2-LS-colon}
	\end{subfigure}
	\hfill
	\begin{subfigure}{0.9\textwidth}
		\centering
		\includegraphics[width=\textwidth]{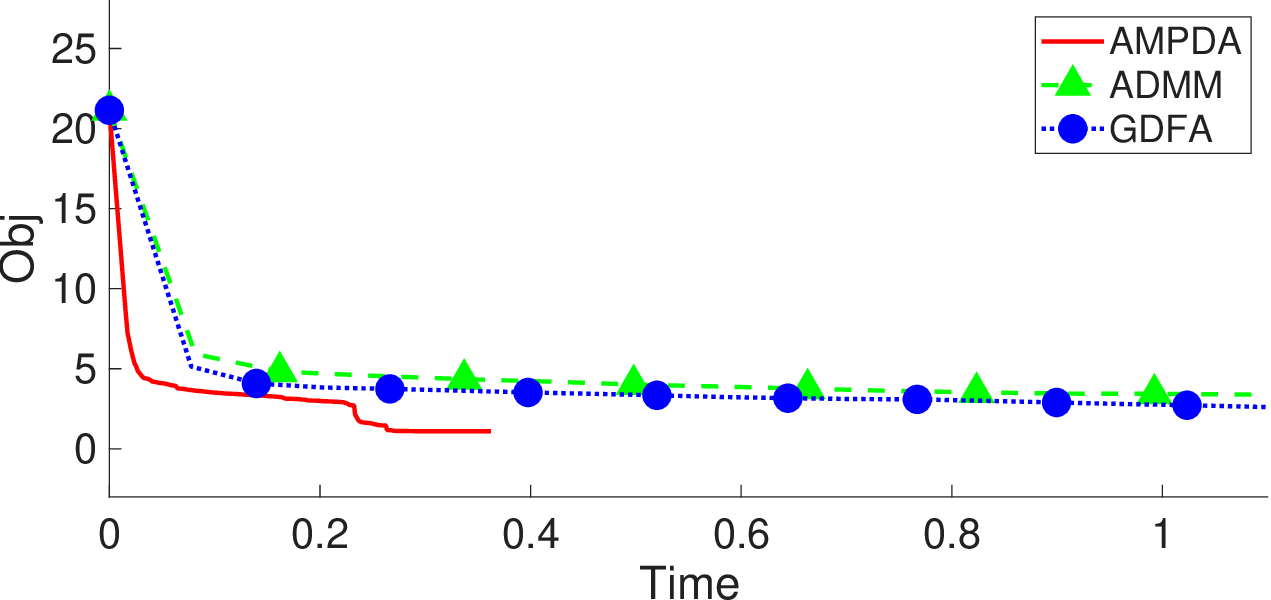}
		\caption{mpg}
		\label{fig:L1L2-LS-mpg7}
	\end{subfigure}
	\caption{
		Objective value versus the computational time during the early stage of the iterative phase (excluding the startup phase) for problem \eqref{problem: L1dL2-LS} on the three datasets.
	}
	\label{fig:L1L2-LS-curves}
\end{figure}

\subsection{Numerical results on the robust signal recovery problem}

In this subsection, we applied AMPDA for solving problems \eqref{problem: L1dL2} and \eqref{problem: L1dSK} with $\mu > 0$. 
Following similar settings to \cite[Section 7.1]{Zeng-Yu-Pong:2021SIAM-OPT}, we generate a sensing matrix $A \in \mathbb{R}^{m \times n}$ with i.i.d. standard Gaussian entries and then normalize each column of $A$. Next, a subset of size $K^{\dagger}$ is chosen uniformly at random from $\{1,2, \cdots, n\}$ and a $K^{\dagger}$-sparse vector $x^{\dagger} \in \mathbb{R}^{n}$ with i.i.d. standard Gaussian entries on this subset is generated as the original true signal. The lower bound $\underline{x}$ and the upper bound $\overline{x}$ are set to be $-\max\{5,\|x^{\dagger}\|_{\infty} \} \One_n$ and $\max\{5,\|x^{\dagger}\|_{\infty} \} \One_n$, respectively, where $\One_n$ denotes the $n$-dimensional vector with all entries being $1$. 
We also produce $\widetilde{z} \in \mathbb{R}^{m}$ as a $\mu^{\dagger}$-sparse vector with $\mu^{\dagger}$ i.i.d. standard Gaussian entries at random (uniformly chosen) positions and set $z=2 \sign(\widetilde{z})$. Finally,
the measurement $b \in \mathbb{R}^{m}$ is generated as $b=A x^{\dagger}-z+0.01 \epsilon$, where $\epsilon \in \mathbb{R}^{m}$ has i.i.d. standard Gaussian entries.

In our numerical tests, we set $(n, m, K^{\dagger}, \mu^{\dagger})=(1280 R, 365 R, 40 R, 5 R)$ with $R \in\{1,2,3,4\}$. For each value of $R$, we generate 50 random instances as described above.
The model parameter $\lambda$ is set to $5$ and $0.5$ for problems \eqref{problem: L1dL2} and \eqref{problem: L1dSK}, respectively. We set $K=\lceil 1.3 K^{\dagger}\rceil$ for problem \eqref{problem: L1dSK}, and $\mu=\lceil 1.3 \mu^{\dagger} \rceil$ for both problems. To assess the performance of AMPDA, we also include a comparison with the gradient descent flow algorithm (GDFA) \cite[Algorithm~1]{Wang-Aujol-Gilboa-Lou:2024IPI}. Using the same notation as in \cite[Algorithm~1]{Wang-Aujol-Gilboa-Lou:2024IPI}, we set the GDFA parameters in this experiment to $\rho = 2$, $\beta = 10^{-6}$, and $\texttt{jMax} = 5$.


\begin{table}[htbp]
	\caption{Results for solving the robust signal recovery problem \eqref{problem: L1dL2}.}
	\begin{tabularx}{\linewidth}{
		>{\raggedright\arraybackslash}p{0.07\linewidth}
		>{\centering\arraybackslash}p{0.08\linewidth}
		>{\centering\arraybackslash}p{0.16\linewidth}
		>{\centering\arraybackslash}p{0.07\linewidth}
		>{\centering\arraybackslash}p{0.19\linewidth}
		>{\centering\arraybackslash}X
	}
		\toprule
		& Alg & Time & Iter & Obj & RecErr \\
		\midrule
		$R=1$ & AMPDA & 0.007(0.012) & 39(7) & 5.087(2.28E-01) & 2.73E-02(7.21E-03) \\
		& GDFA  & 0.338(0.130) & 32(5) & 5.087(2.29E-01) & 2.73E-02(7.14E-03) \\
		&       &             &       &                &                  \\
		$R=2$ & AMPDA & 0.049(0.008) & 50(6) & 7.260(1.98E-01) & 2.29E-02(3.54E-03) \\
		& GDFA  & 1.175(0.082) & 36(3) & 7.260(1.98E-01) & 2.29E-02(3.51E-03) \\
		&       &             &       &                &                  \\
		$R=3$ & AMPDA & 0.169(0.025) & 56(6) & 8.901(2.03E-01) & 2.17E-02(1.77E-03) \\
		& GDFA  & 2.823(0.158) & 41(3) & 8.901(2.03E-01) & 2.16E-02(1.76E-03) \\
		&       &             &       &                &                  \\
		$R=4$ & AMPDA & 0.336(0.032) & 64(6) & 10.359(1.99E-01) & 2.08E-02(1.47E-03) \\
		& GDFA  & 5.375(0.287) & 45(3) & 10.359(1.99E-01) & 2.07E-02(1.49E-03) \\
		\bottomrule
	\end{tabularx}
	\label{Table: exp result of L1dL2}%
\end{table}%

\begin{table}[htbp]
	\caption{Results for solving the robust signal recovery problem \eqref{problem: L1dSK}.}
	\begin{tabularx}{\linewidth}{
			>{\raggedright\arraybackslash}p{0.07\linewidth}
			>{\centering\arraybackslash}p{0.08\linewidth}
			>{\centering\arraybackslash}p{0.16\linewidth}
			>{\centering\arraybackslash}p{0.07\linewidth}
			>{\centering\arraybackslash}p{0.19\linewidth}
			>{\centering\arraybackslash}X
		}
		\toprule
		& Alg & Time & Iter & Obj & RecErr \\
		\midrule
		$R=1$ & AMPDA & 0.004(0.001)   & 30(4)   & 1.008(9.20E-04)  & 1.77E-02(5.25E-03) \\
		& GDFA  & 1.098(0.108)   & 143(14) & 1.008(1.09E-03) & 1.78E-02(5.67E-03) \\
		&       &               &          &                &                    \\
		$R=2$ & AMPDA & 0.040(0.007)   & 42(5)   & 1.014(7.50E-04) & 1.73E-02(1.81E-03) \\
		& GDFA  & 5.314(0.559)   & 179(16) & 1.014(7.82E-04) & 1.59E-02(1.96E-03) \\
		&       &               &          &                &                    \\
		$R=3$ & AMPDA & 0.153(0.020)   & 54(7)   & 1.019(1.13E-03) & 1.97E-02(1.18E-03) \\
		& GDFA  & 16.403(2.106)  & 248(33) & 1.019(1.00E-03) & 1.98E-02(1.22E-03) \\
		&       &               &          &                &                    \\
		$R=4$ & AMPDA & 0.366(0.044)   & 73(9)   & 1.023(9.53E-04) & 2.09E-02(1.10E-03) \\
		& GDFA  & 36.253(5.024)  & 320(48) & 1.023(9.00E-04) & 2.11E-02(1.07E-03) \\
		\bottomrule
	\end{tabularx}
	\label{Table: exp result of L1dSK}%
\end{table}%

The computational results are summarized in Table \ref{Table: exp result of L1dL2} and Table \ref{Table: exp result of L1dSK}. For each method, we report the average computational time in seconds (Time), the number of iterations (Iter), the objective value (Obj), and the recovery error $\text{RecErr}=\|\widehat{x}-x^{\dagger}\|_{2} /\|x^{\dagger}\|_{2}$, where  $\widehat{x}$ denotes the solution returned by the algorithm. In addition, the standard deviation over the 50 random instances is reported in parentheses next to each corresponding average. 
From these results, one can observe that AMPDA achieves objective values and recovery errors comparable to those of GDFA, while requiring substantially less computational time across all problem sizes $R\in\{1,2,3,4\}$. 
Furthermore, the standard deviations indicate that the performance of AMPDA is consistently stable over all test instances.
These results demonstrate the efficiency of the proposed AMPDA for solving the scale-invariant sparse signal recovery models \eqref{problem: L1dL2} and \eqref{problem: L1dSK}.

\appendix
\section{The proof of \eqref{eq: z08072230}}\label{Appendix: subgradient of Q}
In order to prove  \eqref{eq: z08072230}, we first show the following proposition, which concerns the generalized product rule of the Fr{\'e}chet subdifferential.

\begin{proposition}\label{Prop: Product Rule of F-subdiff}
	Let $\varphi_{l}, \varphi_{c}, \varphi_{r}: \mathbb{R}^{n} \to  (-\infty,+\infty]$, and
	define $\Phi: \mathbb{R}^{n} \rightarrow(-\infty,+\infty]$ as follows:
	\begin{equation*}
		\Phi(x)=
		\begin{cases}
			\varphi_{l}(x)(\varphi_{c}(x)+\varphi_{r}(x)), & \text { if } x \in \dom(\varphi_{l}) \cap \dom(\varphi_{c}) \cap \dom(\varphi_{r}), \\
			+\infty, &\text{else.}
		\end{cases}
	\end{equation*}
	Let $\bar{x} \in \dom(\varphi_{l}) \cap \dom(\varphi_{c}) \cap \dom(\varphi_{r})$,
	and suppose the following three conditions hold:
	\begin{enumerate}[\upshape(\romannumeral 1 )]
		\item The $\varphi_{l}$ is nonnegative and locally Lipschitz continuous around $\bar{x}$ relative to $\dom(\varphi_{l})$;
		\item The $\varphi_{c}$ is convex and $\partial \varphi_{c}(\bar{x}) \neq \emptyset$;
		\item The $\varphi_{r}$ is continuous at $\bar{x}$ relative to $\dom(\varphi_{r})$.
	\end{enumerate}
	Then, it holds that
	\begin{equation*}
		\widehat{\partial} \Phi(\bar{x}) \supseteq \widehat{\partial} \psi(\bar{x})+\varphi_{l}(\bar{x}) \partial \varphi_{c}(\bar{x}),
	\end{equation*}
	where $\psi: \mathbb{R}^{n} \to(-\infty,+\infty]$ is defined as:
	\begin{equation*}
		\psi(x)= 
		\begin{cases}
			(\varphi_{c}+\varphi_{r})(\bar{x}) \varphi_{l}(x)+\varphi_{l}(\bar{x}) \varphi_{r}(x), & \text { if } x \in \dom(\varphi_{l}) \cap \dom(\varphi_{r}), \\ 
			+\infty, & \text {else.}
		\end{cases}
	\end{equation*}
\end{proposition}

\begin{proof}
	From the definition of the Fr{\'e}chet subdifferential of $\Phi$ at $\bar{x}$, we have
	\begin{equation}\label{eq: z1218eq1}
		\widehat{\partial} \Phi(\bar{x})=
		\left\{ v:
		\Liminf_{\substack{x \to \bar{x} \\ x \neq \bar{x} \\ x \in \dom(\Phi)}}
		\frac{\Phi(x)-\Phi(\bar{x})-\langle v, x-\bar{x}\rangle}{\|x-\bar{x}\|_2}
		\geq 0
		\right\}.
	\end{equation}
	Let $\bar{y} \in \partial \varphi_{c}(\bar{x})$. Due to $\varphi_{l} \geq 0$ around $\bar{x}$, it holds for all $x\in \dom(\Phi)$ that
	\begin{align}\label{eq: z1218eq3}
		&\Phi(x)-\Phi(\bar{x})-\langle v, x-\bar{x}\rangle\\
		&\geq
		\varphi_{l}(x)(\varphi_c(\bar{x}) + \langle \bar{y}, x-\bar{x} \rangle + \varphi_r(x))
		-(\varphi_{l}(\varphi_{c}+\varphi_{r}))(\bar{x}) 
		-\langle v, x-\bar{x}\rangle \notag \\
		&=
		\psi(x)-\psi(\bar{x})
		- \langle v-\varphi_{l}(x) \bar{y}, x-\bar{x} \rangle 
		+ r(x ), \notag
	\end{align}
	where $r(x ) = (\varphi_{l}(x)-\varphi_{l}(\bar{x}))(\varphi_{r}(x)-\varphi_{r}(\bar{x}))$. It follows that
	\begin{equation}\label{eq: 1218eq4}
		\lim_{\substack{x \to \bar{x} \\ x \neq \bar{x} \\ x \in \dom(\Phi)}}  \frac{|r(x)|}{\|x-\bar{x}\|_{2}}=0, 
	\end{equation}
	owing to the Lipschitz continuity of $\varphi_{l}$ around $\bar{x}$ relative to $\dom(\varphi_{c})$, and the continuity of $\varphi_{r}$ at $\bar{x}$ relative to $\dom(\varphi_{r})$. 
	Combining \eqref{eq: z1218eq1}, \eqref{eq: z1218eq3}, and \eqref{eq: 1218eq4}, we have
	\begin{align}\label{eq: z12241555}
		\widehat{\partial} \Phi(\bar{x})
		&\supseteq
		\left\{ v:
		\Liminf_{\substack{x \to \bar{x} \\ x \neq \bar{x} \\ x \in \dom(\varphi_{l}) \cap \dom(\varphi_{r})}}
		\frac{\psi(x)-\psi(\bar{x})	- \langle v-\varphi_{l}(x) \bar{y}, x-\bar{x} \rangle}{\|x-\bar{x}\|_2}
		\geq 0
		\right\}\\
		&=
		\widehat{\partial} \psi(\bar{x})+\varphi_{l}(\bar{x})  \bar{y},\notag
	\end{align}
	for each $\bar{y}\in\partial \varphi_{c}(\bar{x})$. Consequently, this proposition is derived from \eqref{eq: z12241555}.
	\hfill $\square$ \end{proof}

Now, we are ready to prove the relation \eqref{eq: z08072230}.
Recall that $(\bar{x}, \bar{y}, \bar{z}) \in(\Omega \cap \mathcal{C}) \times \dom(\partial g^{*}) \times \dom(h_{2}^{*})$, and $\bar{c}=1 / g(\bar{x})>0$ in \eqref{eq: z08072230}. According to \eqref{eq: def Q}, $Q$ around $(\bar{x}, \bar{y}, \bar{z}, \bar{c})$ can be represented as
\begin{equation}\label{eq: z1218eqA3}
	Q(x, y, z, c)= c^2 f_{\mathcal{C}}(x)\Big(g^{*}(y)+2/c- \langle x, y\rangle\Big)+h_{1}(x)+h_{2}^{*}(z)-\langle x, z\rangle,
\end{equation}
if $(x, y, z, c) \in \mathcal{C} \times \dom( g^{*}) \times \dom(h_2^{*}) \times (0,+\infty)$, and $Q(x, y, z, c)=+\infty$ otherwise. 
We define $q:\mathbb{R}^n\times\mathbb{R}^n\times\mathbb{R}^n\times\mathbb{R}\to (-\infty,+\infty]$ as	
\begin{equation}\label{eq: z12190000}
	q(x, y, z, c) := c^2 f_{\mathcal{C}}(x)\left(g^{*}(y)+2/c- \langle x, y\rangle\right)
\end{equation}
if $(x, y, c) \in \mathcal{C} \times \dom( g^{*})\times (0,+\infty)$, and $q(x, y, z, c):=+\infty$ otherwise. Then, we derive from \eqref{eq: z1218eqA3} that
\begin{equation}\label{eq: z1218eqA4}
	\widehat{\partial} Q(\bar{x}, \bar{y}, \bar{z}, \bar{c})=
	\widehat{\partial} q(\bar{x}, \bar{y}, \bar{z}, \bar{c}) +		
	\left[\begin{array}{c}
		\nabla h_{1}(\bar{x})-\bar{z} \\
		0_n \\
		\widehat{\partial} h_{2}^{*}(\bar{z})-\bar{x} \\
		0
	\end{array}\right],
\end{equation}
since  $h_{1}(x)-\langle x, z\rangle$ is differentiable and $h_{2}^*$ is separable from $q$.
With the help of Proposition \ref{Prop: Product Rule of F-subdiff}, it follows from \eqref{eq: z12190000} that
\begin{equation}\label{eq: z1218eq7}
	\widehat{\partial} q(\bar{x}, \bar{y}, \bar{z}, \bar{c})\supseteq
	\widehat{\partial} q_2(\bar{x}, \bar{y}, \bar{z}, \bar{c})+
	\left[\begin{array}{c}
		0_n \\
		\bar{c}^2 f_{\mathcal{C}}(\bar{x}) \partial g^*(\bar{y}) \\
		0_n \\
		0
	\end{array}\right],
\end{equation}
where $q_{2}$ is given by
\begin{equation}\label{eq: z1218eq5}
	q_{2}(x, y, z, c):=(g^{*}(\bar{y})+2 / \bar{c}-\langle\bar{x}, \bar{y}\rangle) c^{2} f_{\mathcal{C}}(x)
	+
	\bar{c}^{2} f_{\mathcal{C}}(\bar{x})\left(2/c-\langle x, y\rangle\right),
\end{equation}
if $(x, c) \in \mathcal{C} \times (0,+\infty)$, and $q_{2}(x, y, z, c):=+\infty$ otherwise. On the right-hand side of \eqref{eq: z1218eq5}, note that $g^{*}(\bar{y})+2 / \bar{c}-\langle\bar{x}, \bar{y}\rangle=g(\bar{x})>0$. Hence, by utilizing Proposition \ref{Prop: Product Rule of F-subdiff} again and noting that $2/c-\langle x, y\rangle$ is differentiable at $(\bar{x}, \bar{y}, \bar{c})$, \eqref{eq: z1218eq5} implies that
\begin{equation}\label{eq: z1218eq8}
	\widehat{\partial} q_2(\bar{x}, \bar{y}, \bar{z}, \bar{c})\supseteq
	\left[\begin{array}{c}
		(g^{*}(\bar{y})+2 / \bar{c}-\langle\bar{x}, \bar{y}\rangle) \bar{c}^{2} \widehat{\partial} f_{\mathcal{C}}(\bar{x}) \\
		0_n \\
		0_n \\
		2(g^{*}(\bar{y})+2 / \bar{c}-\langle\bar{x}, \bar{y}\rangle) f_{\mathcal{C}}(\bar{x}) \bar{c}
	\end{array}\right]
	+
	\bar{c}^{2} f_{\mathcal{C}}(\bar{x})
	\left[\begin{array}{c}
		-\bar{y} \\
		-\bar{x} \\
		0_n \\
		-2 / \bar{c}^{2}
	\end{array}\right].
\end{equation}
Combining \eqref{eq: z1218eqA4}, \eqref{eq: z1218eq7}, and \eqref{eq: z1218eq8}, we conclude that \eqref{eq: z08072230} holds.

\section{The initial point $x^0$ given by \eqref{eq: def x0 in exp} satisfies the condition \eqref{eq: requirement of x0}}\label{Appendix:x0}
Firstly, we shall show $ \|a_{i_{\star}}\|_{2}^{2}-\|T  a_{i_{\star}}\|_{2}^2 \neq 0$, ensuring  that the $ \hat{\theta}$ given by \eqref{eq: z11151646} is well-defined. 
Note that $(b-Tb)^{\top}A \neq 0$, since $A$ has full row-rank and $b\neq Tb$ follows from the fact that $b\notin\mathcal{S}_{\mu}$. Then, according to the definition of $i_{\star}$, we have $|(b-Tb)^{\top}a_{i_{\star}}|>0$, which along with the fact $b^{\top}(a_{i_{\star}} -Ta_{i_{\star}})  = (b-Tb)^{\top}a_{i_{\star}}$ implies that $a_{i_{\star}} -Ta_{i_{\star}}\neq 0$. Hence, we deduce that $\|a_{i_{\star}}\|^2_2 -\|Ta_{i_{\star}}\|^2_2 > 0$ and $\hat{\theta}\neq 0$. Secondly, due to $\underline{x}<0<\overline{x}$ and $\hat{\theta}\neq 0$, the $x^0$ given by \eqref{eq: def x0 in exp} is feasible for problems \eqref{problem: L1dL2} and \eqref{problem: L1dSK}, that is, $x^0\neq 0$ and $\underline{x}\leq x^0 \leq \overline{x}$ hold.
Thirdly, owing to $0 \in[\underline{x}_{i_{\star}}, \overline{x}_{i_{\star}}]$ and $x_{i_{\star}}^{0} \neq 0$, we deduce from \eqref{eq: def x0 in exp} that $(x_{i_{\star}}^{0}-\hat{\theta})^{2}<\hat{\theta}^{2}$, or equally, $(x_{i_{\star}}^{0}-2 \hat{\theta}) x_{i_{\star}}^{0}<0$. Note that
$
\|\mathcal{T}_{\mu}(Ax-b)\|^2_2 = \max\left\{ \|E(Ax-b)\|^2_2: E\in\{0,1 \}^{m\times m} \text{ is diagonal,}~\|E\|_0\leq \mu \right\},
$
where $\|\cdot\|_{0}$ counting the number of nonzero elements in the matrix. Then, 
we have
\begin{equation*}
	\begin{aligned}
		&F(x^{0}) \leq  \frac{\|x^{0}\|_{1}}{\|x^{0}\|_{2}}+\frac{\lambda}{2}\|A x^{0}-b\|_{2}^{2}-\frac{\lambda}{2}\|T  (A x^{0}-b)\|_{2}^{2} \\
		&=  1+\frac{\lambda}{2}\left(\|a_{i_{\star}}\|_{2}^{2}-\|T  a_{i_{\star}}\|_{2}^{2}\right)(x_{i_{\star}}^{0}-2 \hat{\theta}) x_{i *}^{0}+\frac{\lambda}{2}\left(\|b\|_{2}^{2}-\|T  b\|_{2}^{2}\right) \\
		&\overset{(*)}{<}  1+\frac{\lambda}{2}(\|b\|_{2}^{2}-\|T  b\|_{2}^{2})=1+\frac{\lambda}{2}\dist^2(b, \mathcal{S}_{\mu}),
	\end{aligned}
\end{equation*}
where $(*)$ follows from the positivity of $\|a_{i_{\star}}\|_{2}^{2}-\|T  a_{i_{\star}}\|_{2}^{2}$ and the negativity of $(x_{i_{\star}}^{0}-2 \hat{\theta}) x_{i_{\star}}^{0}$. This means that $x^0$, defined in \eqref{eq: def x0 in exp}, satisfies the condition \eqref{eq: requirement of x0}.

\quad \\

\noindent\textbf{Funding}~~
	The work of Na Zhang was supported in part by the National Natural Science Foundation of China under grant 12271181, by the Guangzhou Basic Research Program
	under grant 2025A04J5240 and by the Basic and Applied Basic Research Foundation of Guangdong Province under grant 2023A1515030046. \\
	The work of Qia Li was supported in part by the National Natural Science Foundation of China under grant 12471098 and the Guangdong Province Key Laboratory of Computational Science at the Sun Yat-sen University (2020B1212060032).\\

\noindent \textbf{Data Availability}~~ The datasets generated or analyzed during the current study are available from the corresponding author on reasonable request.

\section*{Declarations}

\textbf{Conflict of Interest}~~ The authors declare that they have no conflict of interest.

\bibliographystyle{spmpsci}      
\bibliography{ZhNaPhD}   

\end{document}